\documentclass[12pt,a4paper,final]{amsart}
\usepackage{amsmath,amsthm,amssymb}
\usepackage{amsfonts}
\usepackage[whole]{bxcjkjatype}
\usepackage[top=25truemm,bottom=25truemm,left=30truemm,right=30truemm]{geometry}% for preprint
\usepackage{newtxtext}
\usepackage[varg]{newtxmath}
\usepackage[mathscr]{eucal}
\usepackage{bxpapersize}
\usepackage{ascmac}
\usepackage{epic, eepic}
\usepackage{graphicx}
\usepackage{mathtools}
\usepackage{indentfirst}
\usepackage{cases}
\usepackage{color}
\usepackage{url}
\usepackage[]{multicol}
\usepackage{bm}
\usepackage[all]{xy}
\usepackage[noadjust]{cite}
\usepackage{showkeys}
\usepackage{enumerate}
\usepackage{hyperref}
\hypersetup{% hyperrefオプションリスト
setpagesize=false,
 bookmarksnumbered=true,%
 bookmarksopen=true,%
 colorlinks=true,%
 linkcolor=blue,
 citecolor=blue,
 urlcolor=magenta,
}

\theoremstyle{plane}

\newtheorem{thm}{Theorem}[section]
\newtheorem{prop}[thm]{Proposition}%[section]
\newtheorem{lem}[thm]{Lemma}%[section]
\newtheorem{cor}[thm]{Corollary}%[section]
\newtheorem{fact}[thm]{Fact}%[section]
\theoremstyle{definition}
\newtheorem{dfn}[thm]{Definition}%[section]
\newtheorem{ex}[thm]{Example}%[section]

\theoremstyle{remark}
\newtheorem{rem}[thm]{Remark}%[section]

\newcommand{\rank}{\operatorname{rank}}

\newcommand{\im}{\operatorname{Im}}

\newcommand{\R}{\mathbb{R}}
\newcommand{\bS}{\mathbb{S}}

\newcommand{\bV}{\bm{V}}
\newcommand{\e}{\bm{e}}

\newcommand{\x}{\bm{x}}
\newcommand{\y}{\bm{y}}

\newcommand{\F}{\mathcal{F}}
\newcommand{\cH}{\mathcal{H}}
\newcommand{\Sig}{\Sigma}
\renewcommand{\phi}{\varphi}
\newcommand{\eps}{\varepsilon}
\newcommand{\what}{\widehat}
\newcommand{\til}{\tilde}
\newcommand{\wtil}{\widetilde}
\newcommand{\inner}[2]{\left\langle{#1},{#2}\right\rangle}

\numberwithin{equation}{section}

\title[Focal surfaces of fronts]
{Focal surfaces of fronts associated to unbounded principal curvatures}
\author[K. Teramoto]{Keisuke Teramoto}
\address{Department of Mathematics, 
Hiroshima University, Higashi-Hiroshima 739-8526, Japan}
\email{kteramoto@hiroshima-u.ac.jp}

\thanks{The author was partially supported by JSPS KAKENHI Grant Numbers JP19K14533, JP22K13914 and JP20H01801, 
and the Japan-Brazil bilateral project JPJSBP1 20190103.}
\subjclass[2020]{53A05, 53A55, 57R45}
\keywords{Front, Focal surface, Singularity, Principal curvature, Sub-parabolic point, Gaussian curvature}
\date{\today}

\begin{document}
%%%%%TEXT START%%%%%

\maketitle

\begin{abstract}
We study focal surfaces of (wave) fronts associated to unbounded principal curvatures 
near non-degenerate singular points of initial fronts. 
We give characterizations of singularities of those focal surfaces 
in terms of types of singularities and geometrical properties of initial fronts. 
Moreover, we investigate behavior of the Gaussian curvature of the focal surface.
%\\
%{\sc R\'{e}sum\'{e}}. 
%Dans cet article, nous \'{e}tudions les surfaces focales des fronts par rapport 
%\`{a} la courbure principale non born\'{e}e pr\`{e}s des points singuliers non d\'{e}g\'{e}n\'{e}r\'{e}s des fronts initiaux.
%Nous donnons des caractérisations des singularit\'{e}s de ces surfaces focales par 
%types de singularit\'{e}s et propri\'{e}t\'{e}s g\'{e}om\'{e}triques des fronts initiaux.
%De plus, nous \'{e}tudions le comportement de la courbure gaussienne de la surface focale.
\end{abstract}
%\tableofcontents

%%%%%SECTION 1%%%%%
\section{Introduction}\label{sec:intro}
{\it Focal surfaces} (or {\it caustics}) of regular surfaces in the Euclidean $3$-space $\R^3$ 
can be characterized by several ways: loci of centers of principal curvature spheres 
of initial surfaces, singular value sets of the normal congruence, 
and the bifurcation set of the family of distance squared functions for instance (cf. \cite{ifrt,ist-line,porteous-normal,bgt-focal}). 
Although initial surfaces have no singular points, 
their focal surfaces have singularities in general (\cite{agv,arnold-redbook,zak,ifrt,porteous-normal,porteous-book}). 
It is known that types of (co)rank one singularities on focal surfaces correspond to 
geometric properties arising from principal curvatures of initial surfaces (\cite{ifrt,porteous-normal,porteous-book}). 
Investigating singularities of focal surfaces, one might have new geometrical properties of surfaces 
(\cite{bw-folding,morris,bgt-crest,bgt-focal,ifrt}). 

On the other hand, there are classes of surfaces with singular points 
called {\it frontals} and ({\it wave}) {\it fronts}. 
These surfaces admit smooth unit normal vector field even at singular points. 
Owing to this property, frontals and fronts can be considered as 
a kind of generalization of regular surfaces (i.e., immersions) in $\R^3$.
Recently, there are several studies on frontals and fronts from differential geometric viewpoint, 
and various geometric invariants at singular points are introduced 
(\cite{ft-framed,fsuy-maximal,ist-flat-ce,ms-ce,msuy,suy-front,os-folded,hs-rhamphoid,nuy-cuspidal,
ot-flatce,hnuy-isom,hhnsuy,fukui2020}). 
Moreover, relation among behavior of the curvatures of frontals and fronts and geometric invariants are considered 
(\cite{msuy,hs-rhamphoid,suy-front,tera3,st-behavior}). 
In particular, for a front, it is known that 
one principal curvature can be extended as a bounded $C^\infty$ function 
near a non-degenerate singular point, and another is unbounded near such a singular point (\cite{mura-ume,tera3,tito-2020}). 
Although one principal curvature of a front is unbounded near a non-degenerate singular point, 
the radius function (i.e., the reciprocal) of it 
can be extended as a $C^\infty$ function near the singular point. 
This fact plays a crucial role to study focal surfaces of fronts with bounded Gaussian curvature 
because the bounded principal curvature vanishes on the set of singular points of such fronts, 
and hence we cannot define corresponding focal surfaces near such points (cf. \cite{mura-ume,ku-flattori,kruy-flatfocal,Roitman-flat}).
%For instance, focal surfaces of fronts with bounded Gaussian curvature such as 
%flat fronts in $\R^3$ are essentially those surfaces 
%with respect to unbounded principal curvatures (cf. \cite{mura-ume,ku-flattori,kruy-flatfocal}).

In this paper, we investigate singularities and geometrical properties of focal surfaces of fronts 
associated to unbounded principal curvatures. 
We note that characterizations of singularities on focal surfaces 
relative to bounded principal curvatures are given in \cite{tera2}. 
Moreover, if the initial front has a cuspidal edge, 
then the focal surface associated to the unbounded principal curvature 
is regular at the corresponding point (\cite{tera2}). 
Thus we consider focal surfaces of fronts with 
singular points of the second kind, 
which are classes of non-degenerate singular points of fronts. 
A swallowtail singularity is a typical example. %of a singular point of the second kind of a front in $3$-space. 
To give characterizations of singularities of the focal surface, 
we recall behavior of principal curvatures and 
principal vectors around a non-degenerate singular point in Section \ref{sec:principal}. 
By observing the unbounded principal curvature and corresponding principal vector, 
we extend a concept of {\it sub-parabolic points} to singular points of the second kind on fronts in Section \ref{sec:subpara}. 
In particular, we show relation between behavior of the Gaussian curvature and sub-parabolic points 
of a front (Propositions \ref{prop:subp-rational} and \ref{prop:subpara}). 

In Section \ref{sec:focal}, we study a focal surface of a front associated to the unbounded principal curvature. 
We give characterizations of singularities on focal surfaces 
by geometrical properties of initial fronts (Theorem \ref{thm:singularity-C} and Proposition \ref{prop:sing-C-deg}). 
Furthermore, we study contact between singular sets of the initial front 
and the focal surface associated to the unbounded principal curvature. 
We characterize types of singularities of the focal surface in terms of the contact order of these curves (Proposition \ref{prop:contact}). 
Finally, we observe that the behavior of the Gaussian curvature of the focal surface. 
Especially, for the case of a singular point of the second kind, 
we give characterization for the rational boundedness of the Gaussian curvature of the focal surface 
by a certain geometrical property of the initial front (Theorem \ref{thm:r-bounded-focal}).% and Corollary \ref{cor:kn-C}). 

%%%%%SECTION 2%%%%%
\section{Preliminaries}\label{sec:prelim}
We recall some notions on fronts. 
For details, see \cite{arnold-redbook,suy-front,krsuy-flat,fsuy-maximal,ifrt,agv}.

\subsection{Fronts}
Let $(\Sig;u,v)$ be a domain in the $(u,v)$-plane $\R^2$. 
Let $f\colon\Sig\to\R^3$ be a $C^\infty$ map. 
Then $f$ is said to be a {\it frontal} if there exists a $C^\infty$ map 
$\nu\colon\Sig\to\bS^2$ such that $\inner{df_q(\bm{X})}{\nu(q)}=0$ for any $q\in\Sig$ and $\bm{X}\in T_q\Sig$, 
where $\bS^2$ is the standard unit sphere in $\R^3$ 
and $\inner{\cdot}{\cdot}$ is the canonical inner product on $\R^3$. 
Moreover, a frontal $f$ is a {\it front} if 
the pair $(f,\nu)\colon\Sig\to\R^3\times\bS^2$ gives an immersion. 
We call $\nu$ a {\it unit normal vector field} or a {\it Gauss map} of $f$. 

We fix a frontal $f$. 
A point $p\in\Sig$ is called a {\it singular point} of $f$ if $f$ is not an immersion at $p$. 
We denote by $S(f)$ the set of singular points of $f$. 
Set a function $\lambda\colon\Sig\to\R$ by 
\begin{equation}\label{eq:lambda}
\lambda(u,v)=\det(f_u,f_v,\nu)(u,v)\quad (f_u=\partial f/\partial u,\ f_v=\partial f/\partial v),
\end{equation} 
where $\det$ is the determinant of $3\times3$ matrices. 
Then one can check that $\lambda^{-1}(0)=S(f)$ holds by the definition. 
We call the function $\lambda$ the {\it signed area density function} of $f$. 

Take a singular point $p\in S(f)$ of a frontal $f$. 
Then $p$ is said to be {\it non-degenerate} (resp. {\it degenerate}) 
if $(\lambda_u(p),\lambda_v(p))\neq(0,0)$ (resp. $(\lambda_u(p),\lambda_v(p))=(0,0)$) holds. 
For a non-degenerate singular point $p$ of $f$, 
there exist an open neighborhood $U(\subset\Sig)$ of $p$ 
and a regular curve $\gamma\colon(-\eps,\eps)\ni t\mapsto\gamma(t)\in U$ $(\eps>0)$ such that 
$\gamma(0)=p$ and $\lambda(\gamma(t))=0$ on $U$ 
by the implicit function theorem. 
Moreover, since a non-degenerate singular point $p$ is a corank one singular point (i.e., $\rank df_p=1$), 
there exists a never-vanishing vector field $\eta$ on $U$ such that 
$df_q(\eta_q)=0$ for any $q\in S(f)\cap U$ ($\eta_q\in T_qU$). 
We call $\gamma$ and $\eta$ the {\it singular curve} for $f$ and a {\it null vector field}, respectively. 
We remark that one can take a null vector field $\eta$ of a front near a corank one singular point $p$. 

A non-degenerate singular point $p$ is said to be of the {\it first kind} 
if $\gamma'=d\gamma/dt$ and $\eta$ are linearly independent at $p=\gamma(0)$. 
Otherwise, it is said to be of the {\it second kind}. 
Let $p$ be a singular point of the second kind of a frontal $f$. 
Then $p$ is said to be {\it admissible} if 
for each open neighborhood $U$ of $p$, 
the intersection $S(f)\cap U$ contains a singular point of the first kind. 
Otherwise, we call $p$ {\it non-admissible}. 
We say that a non-degenerate singular point $p$ is {\it admissible} 
if $p$ is either of the first kind or the admissible second kind.

\begin{dfn}
Let $f\colon(\Sig,p)\to\R^3$ be a $C^\infty$ map germ and $p$ a singular point of $f$. 
Then 
\begin{enumerate}
\item $f$ is a {\it cuspidal edge} at $p$ 
if $f$ is $\mathcal{A}$-equivalent to the germ $(u,v)\mapsto(u,v^2,v^3)$ at the origin. 
\item $f$ is a {\it swallowtail} at $p$ 
if $f$ is $\mathcal{A}$-equivalent to the germ $(u,v)\mapsto(u,4v^3+2uv,3v^4+uv^2)$ at the origin.
\item $f$ is a {\it cuspidal butterfly} at $p$ 
if $f$ is $\mathcal{A}$-equivalent to the germ $(u,v)\mapsto(u,5v^4+2uv,4v^5+uv^2)$ at the origin.
\item $f$ is a {\it cuspidal lips} at $p$ if $f$ is $\mathcal{A}$-equivalent to the germ 
$(u,v)\mapsto(u,3v^4+2u^2v^2,v^3+u^2v)$ at the origin.
\item $f$ is a {\it cuspidal beaks} at $p$ if $f$ is $\mathcal{A}$-equivalent to the germ 
$(u,v)\mapsto(u,3v^4-2u^2v^2,v^3-u^2v)$ at the origin.
\end{enumerate}
Here, two map germs $f_1,f_2\colon(\R^2,0)\to(\R^3,0)$ are {\it $\mathcal{A}$-equivalent} 
if there exist diffeomorphism germs $\phi\colon(\R^2,0)\to(\R^2,0)$ on the source 
and $\Phi\colon(\R^3,0)\to(\R^3,0)$ on the target 
such that $f_2=\Phi\circ f_1\circ \phi^{-1}$. 
\end{dfn}

A cuspidal edge, a swallowtail and a cuspidal butterfly are non-degenerate front singular points. 
Moreover, a cuspidal edge and a swallowtail are generic singularities of fronts in $\R^3$, 
and a cuspidal lips/beaks and a cuspidal butterfly are generic corank one singularities of 
one-parameter bifurcation of fronts (cf. \cite{agv,zak}). 
Further, a cuspidal edge is a singular point of the first kind, and 
a swallowtail and a cuspidal butterfly are of the admissible second kind (cf. \cite{msuy}). 
%Thus these are admissible singularities of fronts. 
For these singular points, the following characterizations are known. 

\begin{fact}[{\cite{krsuy-flat,suy-ak,is-mandala,ist-horoflat}}]\label{fact:crit}
Let $f\colon\Sig\to\R^3$ be a front and $p\in\Sig$ a corank one singular point of $f$. 
Then we have the following.
\begin{enumerate}
\item $f$ is a cuspidal edge at $p$ if and only if $\eta\lambda(p)\neq0$.
\item $f$ is a swallowtail at $p$ if and only if $d\lambda(p)\neq0$, $\eta\lambda(p)=0$ and $\eta\eta\lambda(p)\neq0$.
\item $f$ is a cuspidal butterfly at $p$ if and only if $d\lambda(p)\neq0$,
$\eta\lambda(p)=\eta\eta\lambda(p)=0$ and $\eta\eta\eta\lambda(p)\neq0$.
\item $f$ is a cuspidal lips at $p$ if and only if $d\lambda(p)=0$ and $\det(\cH_{\lambda})(p)>0$, 
that is, $\lambda$ has a Morse type singularity of index zero or two at $p$.
\item $f$ is a cuspidal beaks at $p$ if and only if $d\lambda(p)=0$, $\eta\eta\lambda(p)\neq0$ and $\det(\cH_{\lambda})(p)<0$, 
that is, $\lambda$ has a Morse type singularity of index one and $\eta\eta\lambda\neq0$ at $p$.
\end{enumerate}
Here, $\lambda$ is the signed area density function of $f$ as in \eqref{eq:lambda}, 
$\eta$ is a null vector field, $\eta\lambda$ means the directional derivative of $\lambda$ in the direction $\eta$ 
and $\det(\cH_{\lambda})$ is the Hessian of $\lambda$. 
\end{fact}

%We note that criteria for other singular points of frontals and fronts are also known (\cite{fsuy-maximal,hks-cmc}). 

\subsection{Geometric invariants} 
We recall geometric invariants of fronts. 
\subsubsection{Geometric invariants of cuspidal edges}
First we consider the case of cuspidal edges. 
Let $f\colon\Sig\to\R^3$ be a front, $\nu$ its Gauss map and $p$ a cuspidal edge of $f$. 
Let $\gamma(t)$ be a singular curve passing through $p$ and $\eta$ a null vector field of $f$. 
Then one can define the following geometric invariants: the {\it singular curvature} $\kappa_s$ (\cite{suy-front}), 
the {\it limiting normal curvature} $\kappa_\nu$ (\cite{suy-front,msuy}), 
the {\it cuspidal curvature} $\kappa_c$ (\cite{msuy}) 
and the {\it cuspidal torsion} $\kappa_t$ (\cite{ms-ce}).
We note that $\kappa_s$ is an intrinsic invariant and its sing has a geometrical meaning (see \cite{suy-front,hhnsuy}). 
It is known that $\kappa_c$ does not vanish 
when $\gamma$ consists of cuspidal edges (\cite{msuy,ms-ce}). 
We remark that these invariants can be defined at singular points of the first kind for frontals but not fronts. 
In such cases, $\kappa_c$ vanishes at non-front singular points (cf. \cite[Proposition 3.11]{msuy}).

The limiting normal curvature $\kappa_\nu$ relates to the boundedness of the Gaussian curvature of a front with a cuspidal edge. 
In fact, the Gaussian curvature $K$ of a front $f$ is bounded on a sufficiently small neighborhood $U$ 
of a cuspidal edge $p$ if and only if $\kappa_\nu$ vanishes along the singular curve $\gamma$ through $p$ (\cite[Theorem 3.9]{msuy}). 
In this case, $K$ can be extended as a $C^\infty$ function on $U$. 
Moreover, the following property holds. 
\begin{fact}[{\cite[Remark 3.19]{msuy},\cite[Theorem 1.9]{fukui2020}}]\label{fact:K-bounded}
Let $f$ be a front in $\R^3$ and $p$ a cuspidal edge. 
Let $K$ be the Gaussian curvature of $f$ defined on the set of regular points of $f$. 
Suppose that $K$ is bounded on a sufficiently small neighborhood $U$ of $p$. 
Then $K$ satisfies
\begin{equation}\label{eq:K-ce}
4K=-4\kappa_t^2-\kappa_s\kappa_c^2
\end{equation}
at $p$.
\end{fact}
%We will see relationships between the Gaussian curvature of a front and of the focal surface.

On the other hand, we can take a coordinate system $(U;u,v)$ around $p$ satisfying the following properties: 
\begin{enumerate}
\item the $u$-axis is the singular curve,
\item $\partial_v$ gives a null vector field, and 
\item there are no singular points other than the $u$-axis.
\end{enumerate}
We call this local coordinate system $(U;u,v)$ {\it adapted} (\cite{msuy,suy-front,krsuy-flat}). 
Moreover, an adapted coordinate system $(U;u,v)$ is called {\it special adapted} 
if the frame $\{f_u,f_{vv},\nu\}$ gives an orthonormal frame along the $u$-axis in addition to the above conditions 
(\cite[Lemma 3.2]{suy-front}, \cite{msuy}). 

On an adapted coordinate system $(U;u,v)$, since $\eta=\partial_v$, $f_v(u,0)=0$ holds. 
Thus there exists a $C^\infty$ map $g\colon U\to\R^3$ such that $f_v=vg$. 
We note that $g$ does not vanish along the $u$-axis since $f_{vv}=g$ holds along the $u$-axis. 
Therefore the pair $\{f_u,g,\nu\}$ gives a frame along $f$. 
Moreover, when we take a special adapted coordinate system $(U;u,v)$ around a cuspidal edge $p$, 
then $\{f_u,g,\nu\}$ gives an orthonormal frame along the $u$-axis. 
In particular, $\nu$ can be taken as $\nu=(f_u\times g)/|f_u\times g|$. 
Using these mappings, we define the following functions:
\begin{align}\label{eq:fundamental-cusp}
\begin{aligned}
\wtil{E}&=\inner{f_u}{f_u},& \wtil{F}&=\inner{f_u}{g},&\wtil{G}&=\inner{g}{g},\\
\wtil{L}&=-\inner{f_u}{\nu_u},&\wtil{M}&=-\inner{g}{\nu_u},&\wtil{N}&=-\inner{g}{\nu_v}.
\end{aligned}
\end{align} 
We remark that $\wtil{E}\wtil{G}-\wtil{F}^2>0$ on $U$. 
Relation between geometric invariants stated above and the functions 
in \eqref{eq:fundamental-cusp} are known (\cite{hhnsuy,tera3}). 

%\begin{fact}[\cite{hhnsuy},\cite{tera3}]
%Take a special adapted coordinate system $(U;u,v)$ centered at a cuspidal edge $p$. 
%Then geometric invariants as in \eqref{eq:inv-pair} can be represented as 
%\begin{equation}\label{eq:inv-ce}
%\kappa_s=-\dfrac{\wtil{E}_{vv}}{2},\quad 
%\kappa_\nu=\wtil{L},\quad \kappa_c=\dfrac{\wtil{N}}{2},\quad 
%\kappa_t=\wtil{M}
%\end{equation}
%along the $u$-axis.
%\end{fact}
%By this fact, the function $\wtil{N}$ does not vanish along the $u$-axis.

\subsubsection{Geometric invariants at singular points of the second kind}
We next consider geometric invariants at a singular point of the admissible second kind of a front. 
Let $p$ be a singular point of the second kind of a front $f$ in $\R^3$. 
Then one can take a local coordinate system $(U;u,v)$ centered at $p$ satisfying 
\begin{enumerate}
\item $f_u(p)=0$,
\item the $u$-axis is the singular curve on $U$, and
\item $|f_v(p)|=1$.
\end{enumerate}
We also call this local coordinate system {\it adapted}. 
Further, if an adapted coordinate system $(U;u,v)$ around a singular point of the second kind $p$ 
satisfies $\inner{f_{uv}}{f_v}(p)=0$, then $(U;u,v)$ is called a {\it strongly adapted coordinate system} 
(cf. \cite[Definitions 4.1 and 4.6]{msuy}). 
%(We take a strongly adapted coordinate system in the following of this paper 
%when we consider a singular point of the second kind.) 

Using an adapted coordinate system $(U;u,v)$ around a singular point of the second kind $p$, 
we define geometric invariants at $p$ as follows:
\begin{equation}\label{eq:inv2}
\kappa_\nu(p)=\lim_{u\to0}\dfrac{\inner{f_{uu}}{\nu}}{|f_u|^2}(u,0),\quad 
\mu_c=\dfrac{-\inner{f_{uv}}{\nu_u}}{|f_{uv}\times f_v|^2}(p).
\end{equation}
$\kappa_\nu(p)$ is the {\it limiting normal curvature} and $\mu_c$ is the {\it normalized cuspidal curvature} at $p$ 
(see \cite[Proposition 2.9 and (4.7)]{msuy}). 
These invariants are related to the boundedness of the Gaussian and the mean curvature near singular points 
of the second kind (see \cite[Propositions 4.2 and 4.3, Theorem 4.4]{msuy})
%We remark that $\mu_c$ does not vanish if $f$ is a front at $p$. 
%Moreover, if $p$ is a swallowtail, we can consider the following invariant in addition above two: 
%\begin{equation}\label{eq:inv3}
%\tau_s=\dfrac{|\det(f_{uu},f_{uuu},\nu)|}{|f_{uu}|^{5/2}}(p).
%\end{equation}
%We call $\tau_s$ the {\it limiting singular curvature} at $p$ (\cite{msuy}). 

Let us take an adapted coordinate system $(U;u,v)$ 
around a singular point of the (admissible) second kind $p$. 
Then one can take a null vector field $\eta$ as 
$$\eta=\partial_u+e(u)\partial_v,$$
where $e(u)$ is a $C^\infty$ function with $e(0)=0$ (\cite{msuy}). 
We note that there exists a positive integer $l$ such that 
$e(0)=e'(0)=\cdots=e^{(l-1)}(0)=0$ and $e^{(l)}(0)\neq0$ if $p$ is of the admissible (cf. \cite[Lemma 2.2]{saji-swallow}).

Since $df(\eta)=\eta{f}=f_u+e(u)f_v=0$ along the $u$-axis, 
there exists a $C^\infty$ map $h\colon U\to\R^3$ such that $\eta{f}=vh$, 
and hence $f_u=vh-e(u)f_v$. 
Using this map $h$, we have $\lambda=\det(f_u,f_v,\nu)=v\det(h,f_v,\nu)$. 
By the non-degeneracy, $\lambda_v(p)\neq0$ holds. 
Thus we have $\det(h,f_v,\nu)(p)\neq0$. 
This implies that $\{h,f_v,\nu\}$ gives a frame along $f$ near $p$. 
We take $\nu$ satisfying $\lambda_v(p)>0$, 
so one may take $\nu$ as 
$$\nu=\dfrac{h\times f_v}{|h\times f_v|}$$
in what follows. 

We define the following functions:
\begin{align}\label{eq:fundamentals}
\begin{aligned}
\what{E}&=\inner{h}{h},& \what{F}&=\inner{h}{f_v},&\what{G}&=\inner{f_v}{f_v},\\
\what{L}&=-\inner{h}{\nu_u},&\what{M}&=-\inner{h}{\nu_v},&\what{N}&=-\inner{f_v}{\nu_v}.
\end{aligned}
\end{align}
We note that $\what{E}\what{G}-\what{F}^2>0$ on $U$. 
Using functions as in \eqref{eq:fundamentals}, 
differentials $\nu_u$ and $\nu_v$ of $\nu$ can be written as follows.

\begin{lem}[{\cite[Lemma 2.8]{tera3}}]\label{lem:weingarten2}
On an adapted coordinate system $(U;u,v)$, 
$\nu_u$ and $\nu_v$ can be written as 
\begin{align}\label{eq:weingarten}
\begin{aligned}
\nu_u&=\frac{\what{F}(v\what{M}-e(u)\what{N})-\what{G}\what{L}}{\what{E}\what{G}-\what{F}^2}h+
\frac{\what{F}\what{L}-\what{E}(v\what{M}-e(u)\what{N})}{\what{E}\what{G}-\what{F}^2}f_v,\\
\nu_v&=\frac{\what{F}\what{N}-\what{G}\what{M}}{\what{E}\what{G}-\what{F}^2}h+
\frac{\what{F}\what{M}-\what{E}\what{N}}{\what{E}\what{G}-\what{F}^2}f_v.
\end{aligned}
\end{align}
\end{lem}

If $f$ is a front at a singular point of the second kind $p$, 
then $\eta\nu\neq0$ along the singular curve $\gamma$. 
We rephrase this condition using functions as in \eqref{eq:fundamentals}.  
By $\eta=\partial_u+e(u)\partial_v$ and Lemma \ref{lem:weingarten2}, 
it follows that 
\begin{equation}\label{eq:eta-nu}
d\nu(\eta)=\nu_u+e(u)\nu_v
=-\dfrac{\what{L}+e(u)\what{M}}{\what{E}\what{G}-\what{F}^2}(\what{G}h-\what{F}f_v)
\end{equation}
holds along the $u$-axis. 
Thus if $f$ is a front around $p$, then $\what{L}+e(u)\what{M}\neq0$ along the $u$-axis, 
in particular, $\what{L}(p)\neq0$. 
Moreover, we have the following.
\begin{lem}[\cite{msuy,tera3}]\label{lem:relation}
Take a strongly adapted coordinate system $(U;u,v)$ 
around a singular point of the $($admissible$)$ second kind $p$ of a front $f$ in $\R^3$.
Then $\kappa_\nu(p)=\what{N}(p)$ and $\mu_c={\what{L}(p)}/{\what{E}}(p)$ hold. 
%Moreover, if $p$ is a swallowtail of $f$, then 
%$$\tau_s=\dfrac{2|\lambda_v(p)|}{|e'(0)|^{1/2}}$$
%holds.
\end{lem}
\begin{proof}
For $\kappa_\nu$, it follows from \cite[Lemma 2.9]{tera3}. 
For $\mu_c$, we have the expression by \eqref{eq:fundamentals} and \cite[(4.13)]{msuy}. 
\end{proof}

\subsubsection{Rational boundedness of the Gaussian curvature}
Let $f\colon\Sig\to\R^3$ be a front and $p\in \Sig$ a non-degenerate singular point of $f$. 
Then the Gaussian curvature $K$ of $f$ is unbounded near $p$ in general. 
In \cite{msuy}, a notion of the {\it rational boundedness} for unbounded functions was introduced. 
(For precise definition and descriptions of the rational boundedness, see \cite[Definition 3.4 and Page 260]{msuy}.)
For the Gaussian curvature of a front, the following assertion holds.

\begin{fact}[{\cite[Corollary C]{msuy}}]\label{fact:rational-bounded}
Let $f\colon\Sig\to\R^3$ be a front and $p\in \Sig$ a non-degenerate singular point of $f$, 
where $\Sig$ is a domain in $\R^2$. 
Then the following statements are equivalent.
\begin{enumerate}
\item The Gaussian curvature of $f$ is rationally bounded at $p$.
\item The limiting normal curvature of $f$ vanishes at $p$.
\item The Gauss map of $f$ has a singularity at $p$.
\end{enumerate}
\end{fact}

We will characterize rational boundedness of the Gaussian curvature of the focal surface 
in terms of certain geometrical property of the initial front. 

\subsection{Second order derivatives of fronts}\label{append1}
Let $f\colon\Sig\to\R^3$ be a front, $\nu$ its unit normal vector and $p$ a singular point of the second kind of $f$. 
Take an adapted coordinate system $(U;u,v)$ around $p$. 
We then prepare expressions of $h_u$, $h_v$ and $f_{vv}$ by the frame $\{h,f_v,\nu\}$ along $f$, 
where $h$ is a $C^\infty$ map defined on $U$ satisfying $df(\eta)=vh$. 
These shall play important role to analyze some geometric properties of focal surfaces of fronts. 
\begin{lem}\label{lem:derivative}
	Let $f\colon\Sig\to\R^3$ be a front, $\nu$ its unit normal vector and $p\in\Sig$ a singular point of the second kind. 
	Take an adapted coordinate system $(U;u,v)$ centered at $p$. 
	Then we have the following:
	\begin{align}\label{eq:derivatives}
		\begin{aligned}
			h_u&=\frac{\what{E}_u\what{G}-2\what{F}A}{2(\what{E}\what{G}-\what{F}^2)}h+
			\frac{2\what{E}A-\what{E}_u\what{F}}{2(\what{E}\what{G}-\what{F}^2)}f_v+\what{L}\nu,\\
			h_v&=\frac{\what{E}_v\what{G}-2\what{F}B}{2(\what{E}\what{G}-\what{F}^2)}h+
			\frac{2\what{E}B-\what{E}_v\what{F}}{2(\what{E}\what{G}-\what{F}^2)}f_v+\what{M}\nu,\\
			f_{vv}&=\frac{2\what{G}(\what{F}_v-B)-\what{F}\what{G}_v}{2(\what{E}\what{G}-\what{F}^2)}h+
			\frac{\what{E}\what{G}_v-2\what{F}\what{F}_v+2\what{F}B}{2(\what{E}\what{G}-\what{F}^2)}f_v+\what{N}\nu,
		\end{aligned}
	\end{align}
	where $A=\inner{h_u}{f_v}$ and $B=\inner{h_v}{f_v}$. 
\end{lem}
\begin{proof}
	Since $\{h,f_v,\nu\}$ is a moving frame along the front $f$, 
	there exist $C^\infty$ functions $A_i,B_i,C_i\colon U\to\R$ $(i=1,2,3)$ such that 
	$$h_u=A_1h+A_2f_v+A_3\nu,\quad h_v=B_1h+B_2f_v+B_3\nu,\quad 
	f_{vv}=C_1h+C_2f_v+C_3\nu.$$
	We first consider for $h_u$. 
	Taking inner products $\inner{h_u}{h}(=\what{E}_u/2)$ and $\inner{h_u}{f_v}(=A)$, we have 
	$$\begin{pmatrix} \frac{\what{E}_u}{2} \\ A \end{pmatrix}=
	\begin{pmatrix} \what{E} & \what{F} \\ \what{F} & \what{G} \end{pmatrix} 
	\begin{pmatrix} A_1 \\ A_2 \end{pmatrix}.$$
	Solving this equation for $A_1$ and $A_2$, 
	we have 
	$$A_1=\frac{\what{E}_u\what{G}-2\what{F}A}{2(\what{E}\what{G}-\what{F}^2)},\quad 
	A_2=\frac{2\what{E}A-\what{E}_u\what{F}}{2(\what{E}\what{G}-\what{F}^2)}.$$
	Moreover, since $\inner{h}{\nu}=0$, we have $\inner{h_u}{\nu}+\inner{h}{\nu_u}=\inner{h_u}{\nu}-\what{L}=0$, 
	and hence $A_3=\inner{h_u}{\nu}=\what{L}$. 
	Thus we have the assertion for $h_u$. 
	For $h_v$ and $f_{vv}$, we can show similarly 
	by using $\what{F}_v=\inner{h}{f_v}_v=\inner{h_v}{f_v}+\inner{h}{f_{vv}}=B+\inner{h}{f_{vv}}$.
\end{proof}

On an adapted coordinate system $(U;u,v)$, since $f_u=vh-e(u)f_v$, $f_{uu}$ and $f_{uv}$ can be written as 
$f_{uu}=vh_u-e'(u)f_v-e(u)f_{uv}$ and $f_{uv}=h+vh_v-e(u)f_{vv}$, respectively. 
Thus we can write $f_{uu}$ and $f_{uv}$ as 
\begin{equation}\label{eq:diff_fuv}
	\begin{aligned}
		f_{uu}&=(vA_1-e(u)(1+vB_1-e(u)C_1))h\\
		&\quad\quad+(vA_2-e'(u)-e(u)(vB_2-e(u)C_2))f_v+(v(\what{L}-e(u)\what{M})+e(u)^2\what{N})\nu,\\
		f_{uv}&=(1+vB_1-e(u)C_1)h+(vB_2-e(u)C_2)f_v+(v\what{M}-e(u)\what{N})\nu,
	\end{aligned}
\end{equation}
where $A_i,B_i,C_i$ $(i=1,2)$ correspond to the coefficient functions in \eqref{eq:derivatives}.
By \eqref{eq:diff_fuv}, it follows that $f_{uu}(p)=0$ and $f_{uv}(p)=h(p)$ hold. 

For functions $A$ and $B$, the following relation holds. 
\begin{lem}\label{lem:ab}
	Under the above setting, we have 
	$$A+e(u)B=-\what{E}+\what{F}_u+e(u)\what{F}_v-\frac{v \what{E}_v}{2},$$
	where $A=\inner{h_u}{f_v}$ and $B=\inner{h_v}{f_v}$. 
\end{lem}
\begin{proof}
	Since $\what{F}=\inner{h}{f_v}$, we have 
	$\what{F}_u=\inner{h_u}{f_v}+\inner{h}{f_{uv}}=A+\inner{h}{f_{uv}}$. 
	By \eqref{eq:derivatives} and \eqref{eq:diff_fuv}, it holds that 
	$$\inner{h}{f_{uv}}=\what{E}+e(u)(B-\what{F}_v)+\frac{v\what{E}_v}{2}.$$
	Thus we have the assertion.
\end{proof}
%Since $e(0)=0$, we have $A=-\what{E}+\what{F}_u$ at $p$ by the above relation. 
%%%%%SECTION 3%%%%%
\section{Principal curvatures and principal vectors}\label{sec:principal}
%\subsection{Principal curvatures and principal vectors}
We recall principal curvatures and principal vectors. 
We focus on a singular point of the second kind in this section. 
For the case of cuspidal edge, we can consider similar properties (see \cite{tera2,tera3}).

Let $f\colon\Sig\to\R^3$ be a front and $p\in\Sig$ a singular point of the second kind. 
Take a strongly adapted coordinate system $(U;u,v)$ around $p$. 
Then we set functions $\kappa_1$ and $\kappa_2$ on the set of regular points of $f$ by 
\begin{equation}\label{eq:princ-second}
\kappa_1=\dfrac{\hat{k}_1+\hat{k}_2}{2v(\what{E}\what{G}-\what{F}^2)},\quad
\kappa_2=\dfrac{\hat{k}_1-\hat{k}_2}{2v(\what{E}\what{G}-\what{F}^2)},
\end{equation}
where 
\begin{align*}
\hat{k}_1&=\what{G}(\what{L}+e(u)\what{M})-2v\what{F}\what{M}+v\what{E}\what{N},\\
\hat{k}_2&=\sqrt{\hat{k}_1^2-4v(\what{E}\what{G}-\what{F}^2)(\what{N}(\what{L}+e(u)\what{M})-v\what{M}^2)}.
\end{align*}
Since $\what{L}+e(u)\what{M}\neq0$ on the $u$-axis, 
$\hat{k}_1$ and $\hat{k}_2$ are $C^\infty$ functions near $p$. 
Moreover, by direct calculations, we have 
$\kappa_1\kappa_2=K$ and $\kappa_1+\kappa_2=2H$ on the set of regular points of $f$, 
where $K$ is the Gaussian curvature and $H$ is the mean curvature of $f$ written as 
\begin{equation}\label{eq:front_GH}
K=\dfrac{\what{N}(\what{L}+e(u)\what{M})-v\what{M}^2}{v(\what{E}\what{G}-\what{F}^2)},\quad 
H=\dfrac{\what{G}(\what{L}+e(u)\what{M})-2v\what{F}\what{M}+v\what{E}\what{N}}{2v(\what{E}\what{G}-\what{F}^2)}
\end{equation}
(see \cite[Pages 268 and 270]{msuy}). 
Thus one can consider $\kappa_1$ and $\kappa_2$ as {\it principal curvatures} of $f$. 
We note that $H$ is unbounded near $p$ (\cite[Corollary 3.5]{suy-front}), 
and hence at least one principal curvature may diverge near $p$. 
Indeed, the following holds. 

\begin{prop}[cf. {\cite{mura-ume,tera3}}]\label{prop:princ}
Let $f$ be a front in $\R^3$ and $p$ a non-degenerate singular point of $f$. 
Then one principal curvature can be extended as a bounded $C^\infty$ function at $p$, 
and another is unbounded near $p$.
\end{prop}

We note that Medina-Tejeda \cite{tito-2020} shows boundedness of principal curvatures of fronts 
with other singularities.

Let us denote by $\kappa$ and $\tilde{\kappa}$ the bounded principal curvature 
and the unbounded principal curvature, respectively. 
Then $\kappa(p)=\kappa_\nu(p)$ holds at an admissible singular point $p$ (\cite[Theorem 3.1]{tera3}). 
Moreover, setting $\hat{\kappa}=\lambda\tilde{\kappa}$, 
where $\lambda$ is the signed area density function of $f$ in \eqref{eq:lambda}, 
$\hat{\kappa}$ is bounded $C^\infty$ function on $U$, 
in particular, $\hat{\kappa}(p)\neq0$ (\cite[Remark 3.2]{tera3}).  
When $p$ is of the second kind, the relation $\hat{\kappa}(p)=\lambda_v(p)\mu_c$ holds. 
Furthermore, we note that the curvature radius function relative to $\tilde{\kappa}$ can be extended as a $C^\infty$ function on $U$ 
because it can be written as $1/\tilde{\kappa}=\lambda/\hat{\kappa}$. 
We denote by $\hat{\rho}$ the curvature radius function $\lambda/\hat{\kappa}$ of 
the unbounded principal curvature. 

We next consider principal vectors $\bV$ and $\wtil{\bV}$ associated	 to $\kappa$ and $\tilde{\kappa}$, respectively. 
Let $\hat{I}$ and $\what{II}$ be $2\times 2$ matrices given by 
\begin{align*}
\hat{I}&=\begin{pmatrix} 
\inner{f_u}{f_u} & \inner{f_u}{f_v}\\
\inner{f_u}{f_v} & \inner{f_v}{f_v}
\end{pmatrix}
=\begin{pmatrix}
v^2\what{E}-2ve(u)\what{F}+e(u)^2\what{G} & v\what{F}-e(u)\what{G}\\
v\what{F}-e(u)\what{G} & \what{G}
\end{pmatrix},\\
\what{II}&=\begin{pmatrix}
-\inner{f_u}{\nu_u} & -\inner{f_u}{\nu_v}\\
-\inner{f_v}{\nu_u} & -\inner{f_v}{\nu_v}
\end{pmatrix}=
\begin{pmatrix}
v\what{L}-e(u)(v\what{M}-e(u)\what{N}) & v\what{M}-e(u)\what{N}\\
v\what{M}-e(u)\what{N} & \what{N}
\end{pmatrix}.
\end{align*}
When $\bV$ and $\wtil{\bV}$ are principal vectors relative to $\kappa$ and $\tilde{\kappa}$ on $U\setminus\{v=0\}$, 
they satisfy $(\what{II}-\kappa\hat{I}~)\bV=\bm{0}$ and $(\what{II}-\tilde{\kappa}\hat{I}~)\wtil{\bV}=\bm{0}$, respectively. 
Solving the equation $(\what{II}-\kappa\hat{I}~)\bV=\bm{0}$, we have 
\begin{equation}\label{eq:p-vec1}
\bV=(-\what{M}+\kappa\what{F},\what{L}-\kappa(v\what{E}-e(u)\what{F})).
\end{equation}
Since $\kappa$ is bounded on $U$ and $\what{L}(p)\neq0$, $\bV$ can be defined on $U$ (\cite[(3.2)]{tera3}). 

On the other hand, we consider the equation $(\what{II}-\tilde{\kappa}\hat{I}~)\wtil{\bV}=\bm{0}$. 
This equation can be modified as 
\begin{equation}\label{eq:p-vec-mat}
\begin{pmatrix}
\lambda\what{L}-\hat{\kappa}(v\what{E}-e(u)\what{F}) & \lambda\what{M}-\hat{\kappa}\what{F}\\
v(\lambda\what{M}-\hat{\kappa}\what{F})-e(u)(\lambda\what{N}-\hat{\kappa}\what{G}) & \lambda\what{N}-\hat{\kappa}\what{G}
\end{pmatrix}
\begin{pmatrix}
\wtil{V}_1 \\ \wtil{V}_2 
\end{pmatrix}
=\begin{pmatrix} 0 \\ 0 \end{pmatrix}
\end{equation}
by multiplying the equation by $\lambda$ and factoring out $v$ from the first row. 
We note that $\hat{\kappa}$ does not vanish at $p$, 
and hence we may take $\wtil{\bV}$ on $U$ as 
\begin{equation}\label{eq:p-vec2}
\wtil{\bV}=(\lambda\what{N}-\hat{\kappa}\what{G},
-v(\lambda\what{M}-\hat{\kappa}\what{F})+e(u)(\lambda\what{N}-\hat{\kappa}\what{G})). 
\end{equation}
Remark that $\bV$ and $\wtil{\bV}$ as in \eqref{eq:p-vec1} and \eqref{eq:p-vec2} are linearly independent on $U$. 
Indeed, when we identify $\wtil{\bV}=(\wtil{V}_1,\wtil{V}_2)=(\wtil{V}_1,-v\what{V}_2+e(u)\wtil{V}_1)$ 
with $\wtil{\bV}=\wtil{V}_1\partial_u+(-v\what{V}_2+e(u)\wtil{V}_1)\partial_v$, 
where we set $\wtil{V}_1=\lambda\what{N}-\hat{\kappa}\what{G}$ and $\what{V}_2=\lambda\what{M}-\hat{\kappa}\what{F}$, 
it holds that 
\begin{equation}\label{eq:vec2}
\wtil{\bV}=\wtil{V}_1\eta-v\what{V}_2\partial_v.
\end{equation} 
Since $\bV$ is not parallel to $\eta$ along the $u$-axis, 
we have that $\bV$ and $\wtil{\bV}$ are linearly independent. 
Moreover, by $df(\eta)(=\eta{f})=vh$, we have 
$df(\wtil{\bV})=v\{(\lambda\what{N}-\hat{\kappa}\what{G})h-(\lambda\what{M}-\hat{\kappa}\what{F})f_v\}$. 
%In particular, $\wtil{\bV}$ is parallel to the null vector field $\eta$ along the $u$-axis. 

We now define $C^\infty$ maps $\x,\y\colon U\to\R^3$ by 
\begin{equation}\label{eq:vec-xy}
\x=-v(\what{M}-\kappa\what{F})h+((\what{L}+e(u)\what{M})-v\kappa\what{E})f_v,\ \ 
\y=(\lambda\what{N}-\hat{\kappa}\what{G})h-(\lambda\what{M}-\hat{\kappa}\what{F})f_v.
\end{equation}
We note that $\x$ and $\y$ satisfy $\x=df(\bV)$ and $df(\wtil{\bV})=v\y$, respectively.

\begin{lem}\label{lem:perp}
Under the above setting, $\x$ and $\y$ given by \eqref{eq:vec-xy} are perpendicular to each other. 
\end{lem}
\begin{proof}
By a straightforward calculation with relations $2\lambda H=\lambda\kappa+\hat{\kappa}$ and $\lambda K=\kappa\hat{\kappa}$, 
where $K$ and $H$ are the Gaussian and the mean curvature, 
we have $\inner{\x}{\y}=0$.
%\footnote{\textcolor{red}{計算メモ：
%$vV_1\wtil{V}_1\what{E}+((V_2-e(u)V_1)\wtil{V}_1-vV_1\wtil{V}_2)\what{F}-(V_2-e(u)V_1)\wtil{V}_2\what{G}=0$. 
%つまり, $vV_1(\wtil{V}_1\what{E}-\wtil{V}_2\what{F})+(V_2-e(u)V_1)(\wtil{V}_1\what{F}-\wtil{V}_2\what{G})=0$. 
%他にも類似の関係式が成り立たないか調べる必要あり.
%}}
\end{proof}
By this lemma, $\{\x,\y,\nu\}$ gives an orthogonal frame along $f$. 
Moreover, when we set $\e_1=\x/|\x|$ and $\e_2=\y/|\y|$, 
then $\{\e_1,\e_2,\nu\}$ is an orthonormal frame along $f$. 
Particularly, the existence of curvature line coordinate systems for fronts
around non-degenerate singular points is known \cite[Lemma 1.3]{mura-ume}.
We note that a different approach to study singular surfaces using moving frame is known (\cite{ft-framed}). 

\begin{rem}\label{rmk:frame-ce}
If $p$ is a cuspidal edge of a front $f$, 
then principal vectors are given by 
\begin{equation*}\label{eq:ce-pdir}
\bV=(\wtil{N}-v\kappa\wtil{G},-\wtil{M}+\kappa\wtil{F}),\quad 
\wtil{\bV}=(v(\lambda\wtil{M}-\hat{\kappa}\wtil{F}),-\lambda\wtil{L}+\hat{\kappa}\wtil{E})
\end{equation*}
by taking a special adapted coordinate system $(U;u,v)$ 
and using functions as in \eqref{eq:fundamental-cusp} (\cite{tera2,tera3}), 
where $\kappa$ is the bounded principal curvature and $\hat{\kappa}=\lambda\til{\kappa}$. 
In this case, we see that 
$$df(\bV)=(\wtil{N}-v\kappa\wtil{G})f_u+v(-\wtil{M}+\kappa\wtil{F})g,\quad 
df(\wtil{\bV})=v\left((\lambda\wtil{M}-\hat{\kappa}\wtil{F})f_u+(-\lambda\wtil{L}+\hat{\kappa}\wtil{E})g\right).$$  
Then setting $\x,\y\colon U\to\R^3\setminus\{0\}$ by 
$\x= df(\bV)$ and $\y=(\lambda\wtil{M}-\hat{\kappa}\wtil{F})f_u+(-\lambda\wtil{L}+\hat{\kappa}\wtil{E})g$, 
one can see $df(\wtil{\bV})=v\y$, and 
it holds that $\inner{\x}{\y}=0$ by a similar calculation as the proof of Lemma \ref{lem:perp}.
Thus we can take an orthonormal frame $\{\e_1,\e_2,\nu\}$ along $f$ 
by setting $\e_1=\x/|\x|$ and $\e_2=\y/|\y|$.
\end{rem}
\begin{rem}
The pair of principal vectors $\bV$ and $\wtil{\bV}$ is a {\it curvature line frame generator} of a front (cf. \cite{st-behavior}). 
Moreover, the pair $\{\e_1,\e_2\}$ is a {\it curvature line frame corresponding to} $\{\bV,\wtil{\bV}\}$.  
\end{rem}

For the unit normal vector $\nu$ of a front $f$, 
the following properties hold.
\begin{lem}\label{lem:rod1}
	Let $f\colon\Sig\to\R^3$ be a front, $\nu$ its unit normal vector and $p$ a singular point of the second kind. 
	Let $\kappa$ and $\tilde{\kappa}$ be the bounded and unbounded principal curvature on a neighborhood $U$ of $p$, respectively. 
	Let $\bV$ and $\wtil{\bV}$ the corresponding principal vectors. 
	Then we have 
	$$d\nu(\bm{V})=-\kappa df(\bm{V}),\quad df(\wtil{\bV})=-\hat{\rho}d\nu(\wtil{\bV}),$$
	where $\hat{\rho}=\lambda/\hat{\kappa}$.
\end{lem}
\begin{proof}
	Let us take an adapted coordinate system $(U;u,v)$ centered at $p$. 
	We first show the relation between $df(\bV)$ and $d\nu(\bV)$. 
	Denote by $\bm{V}=V_1\partial_u+V_2\partial_v$ the principal vector 
	relative to the bounded principal curvature $\kappa$ (see \eqref{eq:p-vec1}). 
	Then $d\nu(\bm{V})$ can be written as 
	$$d\nu(\bm{V})=V_1\nu_u+V_2\nu_v=(V_1X_1+V_2Y_1)h+(V_1X_2+V_2Y_2)f_v$$
	by Lemma \ref{lem:weingarten2}, 
	where $X_i,Y_i$ ($i=1,2$) are functions in \eqref{eq:weingarten}. 
	By direct computations, we have 
	$V_1X_1+V_2Y_1=v\kappa(\what{M}-\kappa\what{F})$ 
	and $V_1X_2+V_2Y_2=-\kappa((\what{L}+e(u)\what{M})-v\kappa\what{E})$. 
	Thus it holds that 
	$$d\nu(\bm{V})=\kappa\left(v(\what{M}-\kappa\what{F})h-((\what{L}+e(u)\what{M})-v\kappa\what{E})f_v\right)
	=-\kappa df(\bV)(=-\kappa\x)$$
	by \eqref{eq:vec-xy}. 
	
	Next we consider the case of $df(\wtil{\bV})$. 
	By \eqref{eq:p-vec2}, we have 
	$$d\nu(\wtil{\bV})=(\lambda\what{N}-\hat{\kappa}\what{G})(\nu_u+e(u)\nu_v)
	-v(\lambda\what{M}-\hat{\kappa}\what{F})\nu_v.$$
	By \eqref{eq:weingarten}, it follows that 
	$$\nu_u+e(u)\nu_v=\frac{v\what{F}\what{M}-\what{G}(\what{L}+e(u)\what{M})}{\what{E}\what{G}-\what{F}^2}h
	-\frac{v\what{E}\what{M}-\what{F}(\what{L}+e(u)\what{M})}{\what{E}\what{G}-\what{F}^2}f_v
	$$
	holds. 
	Thus on the set of regular points $U\setminus\{v=0\}$, we have 
	$$d\nu(\wtil{\bV})=v\hat{\kappa}\left(-(\what{N}-\til{\kappa}\what{G})h+(\what{M}-\til{\kappa}\what{F})f_v\right).$$
	By multiplying $\lambda$, we obtain 
	$$\lambda d\nu(\wtil{\bV})=v\hat{\kappa}\left(-(\what{N}-\hat{\kappa}\what{G})h+(\what{M}-\hat{\kappa}\what{F})f_v\right)
	=-v\hat{\kappa}\y=-\hat{\kappa}df(\wtil{\bV}).$$
	Since $\hat{\kappa}\neq0$ on $U$, we have 
	$$df(\wtil{\bV})=-\hat{\rho}d\nu(\wtil{\bV})$$
	on $U$. 
\end{proof}
This lemma seems to correspond to the Rodrigues' Theorem (cf. \cite[Theorem 10.2]{porteous-book}). 
We note that if $f$ is a front at a singular point of the second kind $p$, 
$d\nu(\wtil{\bV})=\wtil{V}_1d\nu(\eta)\neq0$ holds at $p$. 
\begin{rem}\label{rem:dn}
	If we take an adapted coordinate system $(U;u,v)$ around a singular point of the second kind $p$ of a front $f$, 
	then the signed area density $\lambda$ of $f$ satisfies $\lambda(u,0)=0$. 
	Thus there exists a $C^\infty$ function $\hat{\lambda}$ on $U$ such that $\lambda=v\hat{\lambda}$. 
	By the non-degeneracy, $\lambda_v(p)\neq0$, and hence $\hat{\lambda}(p)\neq0$. 
	Thus we have 
	$$d\nu(\wtil{\bV})=-\frac{\hat{\kappa}}{\hat{\lambda}}\y$$
	on $U$. 
	Moreover, it holds that 
	$$d\nu(\bV)=-\kappa\x=-\kappa|\x|\e_1,\quad 
	d\nu(\wtil{\bV})=-\frac{\hat{\kappa}}{\hat{\lambda}}\y=-\frac{\hat{\kappa}}{\hat{\lambda}}|\y|\e_2,$$
	where $\e_1=\x/|\x|$ and $\e_2=\y/|\y|$. 
\end{rem}
%%%%%SECTION 4%%%%%
\section{Sub-parabolic points}\label{sec:subpara}
For a regular surface, one can define a ridge point and a sub-parabolic point of the surface 
by using principal curvatures and corresponding principal vectors. 
They are related to singularities and geometrical properties of focal surfaces (cf. \cite{porteous-normal,bw-folding,morris}). 
In this section, we investigate a sub-parabolic point of a front. 
First, we give definitions of a ridge point and a sub-parabolic point for fronts 
using the principal curvatures and corresponding principal vectors. 
\begin{dfn}
Let $\wtil{\bV}$ (resp. $\bV$) be a principal vector 
with respect to the unbounded principal curvature $\tilde{\kappa}$ 
(resp. the bounded principal curvature $\kappa$)
of a front $f$ with a non-degenerate singular point $p$. 
Then $p$ is a {\it sub-parabolic point} (resp. a {\it ridge point}) of $f$ with respect to $\kappa$ 
if the directional derivative $\wtil{\bV}\kappa$ (resp. ${\bV}\kappa$) of $\kappa$ in the direction $\wtil{\bV}$ 
(resp. ${\bV}$) vanishes at $p$. 
\end{dfn}
 
By the property of the principal vector relative to the unbounded principal curvature, 
the following characterization for a sub-parabolic point holds.  
\begin{prop}\label{prop:subpara0}
Let $f\colon\Sig\to\R^3$ be a front and $p\in\Sig$ a non-degenerate singular point. 
Then $p$ is also a sub-parabolic point of $f$ if and only if $\eta\kappa=0$ at $p$, 
where $\eta\kappa$ means the directional derivative of the bounded principal curvature $\kappa$ of $f$ 
in the direction a null vector field $\eta$. 
\end{prop}
\begin{proof}
	Let us take an adapted coordinate system $(U;u,v)$ around $p$. 
	Then the principal vector $\wtil{\bV}$ associated to the unbounded principal curvature is parallel to $\eta$ 
	along the $u$-axis by \eqref{eq:vec2} and Remark \ref{rmk:frame-ce}. 
	Thus we have the assertion by the definition of a sub-parabolic point. 
\end{proof}
For the case of a cuspidal edge $p$ of a front $f$, the condition for $p$ being a sub-parabolic point 
is characterized by the geometric invariants (see \cite[Proposition 2.8]{tera2}). 
On the other hand, for the case of a singular point of the second kind, we have the following 
characterization relating the behavior of the Gaussian curvature of a front. 

\begin{prop}\label{prop:subp-rational}
	Let $f\colon\Sig\to\R^3$ be a front and $p\in\Sig$ a singular point of the second kind. 
	If the Gaussian curvature $K$ of $f$ is rationally continuous, 
	then $p$ is a sub-parabolic point. 
\end{prop}	
For a notion of the {\it rational continuity}, see \cite[Definition 3.4]{msuy}. 
\begin{proof}
	Let us take an adapted coordinate system $(U;u,v)$ centered at $p$. 
	Then we define a co-vector $\omega_\nu(p)$ by 
	$\omega_\nu(p)=\hat{\kappa}_\nu'(p)du\in T_p^{\ast}\Sig,$
	where $\hat{\kappa}_\nu'(p)=d\kappa_\nu(u)/du|_{u=0}$. 
	We note that $\hat{\kappa}_\nu'(p)$ depends on the parameter $u$ of the singular curve, 
	but $\omega_\nu$ does not depend on $u$ (see \cite[Page 270]{msuy}). 
	Since $\kappa_u(p)=\hat{\kappa}_\nu'(p)$ holds, 
	the co-vector $\omega_\nu(p)$ may be written as $\omega_\nu(p)=\kappa_u(p)du$. 
	Moreover, by Proposition \ref{prop:subpara0} and $\eta=\partial_u$ at $p$, 
	$\omega_\nu(p)=0$ if and only if $p$ is a sub-parabolic point of $f$. 
	On the other hand, it is known that the Gaussian curvature $K$ of $f$ is rationally continuous at $p$ 
	if and only if $\kappa_\nu(p)=\omega_{\nu}(p)=0$ (\cite[Theorem 4.4]{msuy}). 
	Thus we have the conclusion.
\end{proof}

We next consider the case that the Gaussian curvature $K$ of a front $f$ 
is bounded near a non-degenerate singular point $p$. 
In such a case, we obtain the following.
\begin{prop}\label{prop:subpara}
Let $f\colon\Sig\to\R^3$ be a front and $p\in\Sig$ a non-degenerate singular point. 
Suppose that the Gaussian curvature $K$ of $f$ is bounded 
on a sufficiently small neighborhood $U$ of $p$. 
\begin{enumerate}
	\item When $p$ is a cuspidal edge, $p$ is a sub-parabolic point of $f$ if and only if $K(p)=0$.
	\item When $p$ is of the second kind, then $p$ is a sub-parabolic point of $f$.
\end{enumerate}
\end{prop}
\begin{proof}
First we show the case of a cuspidal edge. 
In this case, it is known that $p$ is a sub-parabolic point of $f$ 
if and only if $4\kappa_t^2+\kappa_s\kappa_c^2=0$ holds at $p$ (\cite[Proposition 2.8]{tera2}). 
On the other hand, the Gaussian curvature satisfies $4K=-(4\kappa_t^2+\kappa_s\kappa_c^2)$ at $p$ 
by Fact \ref{fact:K-bounded}. 
This shows the first assertion. 

We next consider the second assertion. 
Let us take an adapted coordinate system $(u,v)$ on $U$. 
Since $K$ is bounded on $U$, the bounded principal curvature $\kappa$ satisfies $\kappa(u,0)=\kappa_\nu(u)=0$. 
Thus there exists a $C^\infty$ function $k\colon U\to\R$ such that $\kappa=vk$. 
Since the principal vector $\wtil{\bm{V}}$ relative to the unbounded principal curvature can be written 
as $\wtil{\bm{V}}=\wtil{V}_1\eta-v\what{V}_2\partial_v$ (see \eqref{eq:vec2}), 
it holds that 
$\wtil{\bm{V}}\kappa=\wtil{V}_1\eta (vk)-v\what{V}_2(k+vk_v)=\wtil{V}_1(vk_u+e(u)(k+vk_v))-v\what{V}_2(k+vk_v)$. 
Therefore we have $\wtil{\bV}\kappa=0$ at $p$ since $e(0)=0$. 
Thus we get the conclusion.
\end{proof}

For the case of a front with a cuspidal edge $p$, 
it is known that the Gaussian curvature of the focal surface 
associated to the unbounded principal curvature vanishes at $p$ 
if and only if $p$ is a sub-parabolic point (\cite[Corollary 3.9]{tera2}). 
For the case of a singular point of the second kind, 
we will show relation between behavior of the Gaussian curvature of the focal surface 
associated to the unbounded principal curvature and a sub-parabolic point 
of the initial front in the next section. 
%%%%%SECTION 5%%%%%
\section{Focal surfaces of fronts}\label{sec:focal}
We consider singularities and geometric properties of focal surfaces of fronts 
associated to unbounded principal curvatures. 
In particular, we focus on the case that the initial front has 
a singular points of the second kind. 
%For the case of cuspidal edges, see \cite[Section 3.4]{tera2}. 

Let $f\colon\Sig\to\R^3$ be a front, $\nu$ its Gauss map 
and $p\in \Sig$ a non-degenerate singular point of $f$. 
Take an adapted coordinate system $(U;u,v)$ centered at $p$. 
Then we consider a map $\F\colon U\times\R\to\R^3$ given by 
$$\F(u,v,w)= f(u,v)+w\nu(u,v).$$
The map $\F$ is called a {\it normal congruence} of $f$ (cf. \cite{ist-line}). 
We consider the singular set and the singular value set of $\F$. 
By direct calculations using the Weingarten formula, 
we have 
$$\det(\F_u,\F_v,\F_w)=(1-w\kappa)(\lambda-\hat{\kappa}w).$$
Thus the set of singular points $S(\F)$ of $\F$ is the union of 
$S_1(\F)=\{(u,v,w)\ |\ 1-\kappa(u,v)w=0\}$ 
and $S_2(\F)=\{(u,v,w)\ |\ \lambda(u,v)-\hat{\kappa}(u,v)w=0\}$. 
Hence the image $\F(S(\F))$ of $S(\F)$ by $\F$ is 
\begin{align*}
\F(S(\F))&=\left\{f(u,v)+\rho(u,v)\nu(u,v)\ |\ (u,v)\in U,\ w=\rho(u,v)\right\}\\
&\quad \cup \left\{f(u,v)+\hat{\rho}(u,v)\nu(u,v)\ |\ (u,v)\in U,\ w=\hat{\rho}(u,v)\right\},
\end{align*}
where $\rho=1/\kappa$ and $\hat{\rho}=\lambda/\hat{\kappa}$. 
We define $C^\infty$ maps $C,\what{C}\colon U\to\R^3$ by 
\begin{equation}\label{eq:focal}
C=f+\rho\nu,\quad 
\what{C}= f+\hat{\rho}\nu.
\end{equation}
Then one can notice immediately that the singular value set of $\F$ coincides with the union of the 
images of $C$ and $\what{C}$: $\F(S(F))=\im{C}\cup\im{\what{C}}$. 
We call $C$ and $\what{C}$ {\it focal surfaces} (or {\it caustics}) of $f$. 
In particular, we call $\what{C}$ the {\it focal surface associated to the unbounded principal curvature} of $f$. 
We remark that $C$ cannot be defined near a singular point 
when the Gaussian curvature of an original front is bounded near the singular point. 
However, we can define and consider $\what{C}$ even if the Gaussian curvature is bounded. 
We note that singularities of $C$ for a front with non-degenerate singular points are studied in \cite{tera2}.
Thus we focus on $\what{C}$ of a front $f$ in the following of this paper. 

\subsection{Singularities of $\what{C}$}
We first investigate singularities of $\what{C}$. 
When a front $f$ has a cuspidal edge $p$, 
then $\what{C}$ is regular at $p$ (\cite[Proposition 3.7]{tera2}). 
Thus we treat the case that a front $f$ has a 
singular point of the second kind $p$ in the following of this subsection. 

\begin{lem}\label{lem:normal}
The map $\y$ as in \eqref{eq:vec-xy} is a normal vector to the focal surface $\what{C}$. 
\end{lem} 
\begin{proof}
Let us take a strongly adapted coordinate system $(U;u,v)$ around $p$.  
By direct calculations, we have 
$$
\what{C}_u=f_u+\hat{\rho}\nu_u+\hat{\rho}_u\nu,\quad 
\what{C}_v=f_v+\hat{\rho}\nu_v+\hat{\rho}_v\nu.
$$
On the other hand, we get 
\begin{align*}
\inner{\y}{f_u}&=(\lambda\what{N}-\hat{\kappa}\what{G})(v\what{E}-e(u)\what{F})
-(\lambda\what{M}-\hat{\kappa}\what{F})(v\what{F}-e(u)\what{G}),\\
\inner{\y}{f_v}&=(\lambda\what{N}-\hat{\kappa}\what{G})\what{F}
-(\lambda\what{M}-\hat{\kappa}\what{F})\what{G},\\
\inner{\y}{\nu_u}&=-(\lambda\what{N}-\hat{\kappa}\what{G})\what{L}+
(\lambda\what{M}-\hat{\kappa}\what{F})(v\what{M}-e(u)\what{N}),\\
\inner{\y}{\nu_v}&=-(\lambda\what{N}-\hat{\kappa}\what{G})\what{M}
+(\lambda\what{M}-\hat{\kappa}\what{F})\what{N}.
\end{align*}
Therefore by $\hat{\rho}\hat{\kappa}=\lambda$ and $\inner{\y}{\nu}=0$, it follows that 
\begin{align*}
\inner{\what{C}_u}{\y}&=
-v(\what{E}\what{G}-\what{F}^2)(\hat{\kappa}-2\lambda H+\lambda\hat{\rho}K)
=-v(\what{E}\what{G}-\what{F}^2)(\hat{\kappa}-(\lambda\kappa+\hat{\kappa})+\lambda\kappa)=0,\\
\inner{\what{C}_v}{\y}&=\lambda(\what{N}\what{F}-\what{M}\what{G})
+\hat{\rho}\hat{\kappa}(\what{M}\what{G}-\what{N}\what{F})=0
\end{align*} 
hold. 
Thus we have the assertion.
\end{proof}

Since $\y$ does not vanish near $p$, 
$\e_2=\y/|\y|$ is a unit normal vector field of $\what{C}$ by Lemma \ref{lem:normal}. 
Thus $\what{C}$ is a frontal near $p$. 
We remark that for the case of cuspidal edge, we have similar properties 
using the map $\y$ defined in Remark \ref{rmk:frame-ce}. 

\begin{lem}\label{lem:sing-C}
Let $f$ be a front in $\R^3$ and $p$ a singular point of the second kind of $f$. 
Then the set of singular points $S(\what{C})$ of $\what{C}$ in \eqref{eq:focal} 
coincides with the zero set of $\wtil{\bV}\hat{\rho}$. 
In particular, $p$ is also a singular point of $\what{C}$. 
\end{lem}
\begin{proof}
Let us take a strongly adapted coordinate system $(U;u,v)$ around $p$. 
We denote by $\lambda^{\what{C}}$ the signed area density function of $\what{C}$ 
given by $\lambda^{\what{C}}=\det(\what{C}_u,\what{C}_v,\e_2)=\inner{\what{C}_u\times\what{C}_v}{\e_2}$. 
We calculate it explicitly. 
Since $\nu$ can be taken as 
$\nu=(h\times f_v)/|h\times f_v|$ 
and $\e_2$ is perpendicular to $\nu$, we need 
$$
\what{C}_u\times \what{C}_v\equiv
\hat{\rho}_vf_u\times\nu+\hat{\rho}\hat{\rho}_v\nu_u\times\nu
+\hat{\rho}_u\nu\times f_v+\hat{\rho}\hat{\rho}_u\nu\times \nu_v \mod \nu.
$$
Here for two $C^\infty$ maps $\alpha,\beta\colon U\to\R^3$, 
$\alpha\equiv\beta\mod \nu$ implies that there exists a $C^\infty$ function $\delta\colon U\to\R$ 
such that $\alpha-\beta=\delta\nu$ holds. 
By the vector triple product and Lemma \ref{lem:weingarten2}, we have 
\begin{align*}
f_u\times \nu&=\dfrac{(v\what{F}-e(u)\what{G})h+(e(u)\what{F}-v\what{E})f_v}{|h\times f_v|},\quad 
f_v\times\nu=\dfrac{\what{G}h-\what{F}f_v}{|h\times f_v|},\\
\nu_u\times\nu&=\dfrac{(e(u)\what{N}-v\what{M})h+\what{L}f_v}{|h\times f_v|},\quad 
\nu_v\times\nu=\dfrac{-\what{N}h+\what{M}f_v}{|h\times f_v|}.
\end{align*}
Therefore, one can see that 
\begin{multline*}
\what{C}_u\times \what{C}_v\equiv
\dfrac{\hat{\rho}_u}{\hat{\kappa}|h\times f_v|}
((\lambda\what{N}-\hat{\kappa}\what{G})h-(\lambda\what{M}-\hat{\kappa}\what{F})f_v)\\
+\dfrac{\hat{\rho}_v}{\hat{\kappa}|h\times f_v|}
((-v(\lambda\what{M}-\hat{\kappa}\what{F})+e(u)(\lambda\what{N}-\hat{\kappa}\what{G}))h\\
\quad+(\lambda\what{L}-\hat{\kappa}(v\what{E}-e(u)\what{F}))f_v)
\mod \nu.
\end{multline*}

On the other hand, for $\wtil{\bV}=(\wtil{V}_1,\wtil{V}_2)$ as in \eqref{eq:p-vec2}, we note that  
$$
(\lambda\what{L}-\hat{\kappa}(v\what{E}-e(u)\what{F}))\wtil{V}_1
+(\lambda\what{M}-\hat{\kappa}\what{F})\wtil{V}_2=0
$$
holds by \eqref{eq:p-vec-mat}. 
Since $\wtil{V}_1\neq0$ near $p$, we have 
$$
(\lambda\what{L}-\hat{\kappa}(v\what{E}-e(u)\what{F}))=
-(\lambda\what{M}-\hat{\kappa}\what{F})\dfrac{\wtil{V}_2}{\wtil{V}_1}.
$$
Hence it follows that 
\[\begin{aligned}
\what{C}_u\times\what{C}_v&\equiv
\dfrac{\hat{\rho}_u\y}{\hat{\kappa}|h\times f_v|}
+\dfrac{\hat{\rho}_v}{\hat{\kappa}|h\times f_v|}
\left(\wtil{V}_2h-(\lambda\what{M}-\hat{\kappa}\what{F})\dfrac{\wtil{V}_2}{\wtil{V}_1}f_v\right)\\
&\equiv \dfrac{\hat{\rho}_u\y}{\hat{\kappa}|h\times f_v|}
+\dfrac{\wtil{V}_2\hat{\rho}_v}{\wtil{V}_1\hat{\kappa}|h\times f_v|}(\wtil{V}_1h-(\lambda\wtil{M}-\hat{\kappa}\wtil{F})f_v)\\
&\equiv \dfrac{(\wtil{\bV}\hat{\rho})\y}{\wtil{V}_1\hat{\kappa}|h\times f_v|} \mod \nu.
\end{aligned}\]
Summing up, we obtain 
\begin{equation}\label{eq:lambda2}
\lambda^{\what{C}}=\dfrac{(\wtil{\bV}\hat{\rho})|\y|}{\wtil{V}_1\hat{\kappa}|h\times f_v|}.
\end{equation}
Thus $S(\what{C})=(\wtil{\bV}\hat{\rho})^{-1}(0)$ holds. 
Further, since $\hat{\rho}_u(p)=\wtil{V}_2(p)=0$, 
$$\wtil{\bV}\hat{\rho}=\wtil{V}_1\hat{\rho}_u+\wtil{V}_2\hat{\rho}_v=0$$
holds at $p$. 
Hence $p$ is also a singular point of $\what{C}$
\end{proof}

\begin{lem}\label{lem:corank}
Let $p$ be a singular point of the second kind of a front $f$ in $\R^3$. 
Then the focal surface $\what{C}$ satisfies $\rank d\what{C}_p=1$. 
Moreover, $\wtil{\bV}$ can be taken as a null vector field $\eta^{\what{C}}$ of $\what{C}$ near $p$. 
\end{lem}
\begin{proof}
Take a strongly adapted coordinate system $(U;u,v)$ centered at $p$. 
By a direct calculation, we see that 
$$\what{C}_u=0,\quad 
\what{C}_v=f_v+\hat{\rho}_v\nu=f_v+\dfrac{\lambda_v}{\hat{\kappa}}\nu
\left(=f_v+\dfrac{1}{\mu_c}\nu\right)\neq0$$
hold at $p$. 
Thus we have the first assertion. 

We next show the second assertion. 
By Lemma \ref{lem:rod1}, we have 
$$d\what{C}(\wtil{\bV})=df(\wtil{\bV})+\hat{\rho}d\nu(\wtil{\bV})+(\wtil{\bV}\hat{\rho})\nu=(\wtil{\bV}\hat{\rho})\nu.$$
Since $S(\what{C})=(\wtil{\bV}\hat{\rho})^{-1}(0)$, $d\what{C}(\wtil{\bV})$ vanishes on $S(\what{C})$. 
This implies that $\wtil{\bV}$ is a null vector field of $\what{C}$. 
\end{proof}

\begin{lem}\label{lem:swallow}
Let $p$ be a singular point of the second kind of a front $f\colon\Sig\to\R^3$. 
Then we have the following:
\begin{enumerate}
\item When $p$ is a swallowtail of $f$, 
then $p$ is a non-degenerate singular point of $\what{C}$.
\item When $p$ is not a swallowtail of $f$, then 
$p$ is a non-degenerate singular point of $\what{C}$ 
if and only if $\bV(\wtil{\bV}\hat{\rho})\neq0$ at $p$, 
where $\bV$ is a principal vector associated to the bounded principal curvature $\kappa$. 
%This condition is equivalent to that $p$ is not a sub-parabolic point of $f$.
\end{enumerate}
\end{lem}
\begin{proof}
Taking a strongly adapted coordinate system $(U;u,v)$ centered at $p$ with $\lambda_v(p)>0$, 
we have 
\begin{align*}
(\wtil{\bV}\hat{\rho})_u&=
(\wtil{V}_1)_u\hat{\rho}_u+\wtil{V}_1\hat{\rho}_{uu}+(\wtil{V}_2)_u\hat{\rho}_v+\wtil{V}_2\hat{\rho}_{uv}
=(\wtil{V}_2)_u\hat{\rho}_v=-e'(0)\hat{\kappa}\dfrac{\lambda_v}{\hat{\kappa}}
=-e'(0)\lambda_v,\\%\neq0
(\wtil{\bV}\hat{\rho})_v&=
(\wtil{V}_1)_v\hat{\rho}_u+\wtil{V}_1\hat{\rho}_{uv}+(\wtil{V}_2)_v\hat{\rho}_v+\wtil{V}_2\hat{\rho}_{vv}
=\wtil{V}_1\hat{\rho}_{uv}=-\hat{\kappa}\hat{\rho}_{uv}
\end{align*}
at $p$ by \eqref{eq:p-vec2}. 
Thus by the definition, $p$ is a non-degenerate singular point of $\what{C}$ if and only if 
$(e'(0),\hat{\rho}_{uv})\neq(0,0)$ at $p$ because $\lambda_v(p)\neq0$ and $\hat{\kappa}(p)\neq0$. 
In particular, the condition $e'(0)\neq0$ holds when $p$ is a swallowtail of the initial front $f$ by Fact \ref{fact:crit}. 
Thus the first assertion holds.

On the other hand, $\bV(\wtil{\bV}\hat{\rho})$ is calculated as 
$$\bV(\wtil{\bV}\hat{\rho})=V_1(\wtil{\bV}\hat{\rho})_u+V_2(\wtil{\bV}\hat{\rho})_v
=e'(0)\lambda_v\what{M}-\hat{\kappa}\hat{\rho}_{uv}\what{L}$$
at $p$ by \eqref{eq:p-vec1}. 
Thus $\bV(\wtil{\bV}\hat{\rho})\neq0$ if and only if $\hat{\rho}_{uv}\neq0$ at $p$ 
when $p$ is not a swallowtail, 
that is, $e'(0)=0$. 
Hence we obtain the second assertion.
\end{proof}
When the initial front $f$ has a cuspidal butterfly at $p$, 
the focal surface $\what{C}$ may have a degenerate singularity at $p$ by this lemma. 

\begin{prop}\label{prop:frontness}
Let $f\colon\Sig\to\R^3$ be a front with a singular point of the second kind $p$. 
Then the focal surface $\what{C}$ as in \eqref{eq:focal} is a front at $p$. 
%Moreover, if $f$ at $p$ is a swallowtail, then $\what{C}$ at $p$ is a cuspidal edge.
\end{prop}
\begin{proof}
We take a strongly adapted coordinate system $(U;u,v)$ centered at $p$. 
Since $\what{C}_u=0$ at $p$ and $p$ is a corank one singular point of $\what{C}$, 
$\what{C}$ is a front at $p$ if and only if $(\e_2)_u(p)\neq0$, 
where $\e_2=\y/|\y|$ and $\y$ is a map given by \eqref{eq:vec-xy}. 
We note that this condition is equivalent to $d\e_2(\wtil{\bV})\neq0$ at $p$ 
because $\wtil{\bV}$ is parallel to $\partial_u$ at $p$ (see \eqref{eq:p-vec2}). 
%Assume that $\hat{\kappa}>0$ near $p$. 
By a direct computation, it holds that 
$$
(\e_2)_u=\dfrac{\y_u|\y|^2-\y\inner{\y_u}{\y}}{|\y|^3}.
%=\dfrac{\hat{\kappa}}{|\y|^2}\left(-\dfrac{\what{E}_u}{2}h\right)
$$
We consider the expression of $\y_u$ at $p$. 
Since $\lambda(p)=\what{F}(p)=\what{G}_u(p)=0$ and $\what{G}(p)=1$, 
we get 
$$\y_u=-\hat{\kappa}_uh-\hat{\kappa}h_u+\hat{\kappa}\what{F}_uf_v$$
at $p$, 
where $h_u$ can be written as in \eqref{eq:derivatives}. 
By Lemma \ref{lem:ab}, $A=\inner{h_u}{f_v}=\what{F}_u-\what{E}$ holds at $p$. 
Thus $\y_u$ satisfies 
$$\y_u=-\left(\hat{\kappa}_u+\dfrac{\hat{\kappa}\what{E}_u}{2\what{E}}\right)h
+\hat{\kappa}\what{E}(f_v-\mu_c\nu)$$
at $p$ since $\mu_c=(\what{L}/\what{E})(p)$ holds by Lemma \ref{lem:relation}. 
Since $\y=-\hat{\kappa}h$ at $p$, we have 
$$\y_u|\y|^2-\y\inner{\y}{\y_u}=\hat{\kappa}^3\what{E}(\what{E}f_v-\what{L}\nu)
=\hat{\kappa}^3\what{E}^2\left(f_v-\mu_c\nu\right)\neq0$$
at $p$. 
Therefore $(\e_2)_u\neq0$ at $p$, and hence $\what{C}$ is a front at $p$.
\end{proof}
If a front $f$ is a cuspidal edge at $p$, 
then $\what{C}$ is regular at $p$, and hence we can also consider $\what{C}$ as a front near $p$. 

Summarizing above results and using the criteria for singularities given by Fact \ref{fact:crit}, 
we have the following characterizations.
\begin{thm}\label{thm:singularity-C}
Let $f\colon\Sig\to\R^3$ be a front and $p\in\Sig$ a singular point of the second kind. 
Then we have the following.
\begin{enumerate}
\item The focal surface $\what{C}$ given by \eqref{eq:focal} is a cuspidal edge at $p$ 
if and only if $\bV(\wtil{\bV}\hat{\rho})\neq0$ and $\wtil{\bV}(\wtil{\bV}\hat{\rho})\neq0$ at $p$. 
In particular, $p$ is a swallowtail of $f$.
\item The focal surface $\what{C}$ is a swallowtail at $p$ 
if and only if $\bV(\wtil{\bV}\hat{\rho})\neq0$, $\wtil{\bV}(\wtil{\bV}\hat{\rho})=0$ and 
$\wtil{\bV}\wtil{\bV}(\wtil{\bV}\hat{\rho})\neq0$ at $p$. 
In particular, $p$ is a cuspidal butterfly of $f$.
\end{enumerate}
\end{thm} 
\begin{proof}
Let us take a strongly adapted coordinate system $(U;u,v)$ centered at $p$. 
By Lemma \ref{lem:corank}, $\wtil{\bV}$ as in \eqref{eq:p-vec2} can be taken as a null vector field of $\what{C}$. 
Moreover, the signed area density function $\lambda^{\what{C}}$ is proportional to 
$\wtil{\bV}\hat{\rho}$ by \eqref{eq:lambda2}. 
Thus $\wtil{\bV}\lambda^{\what{C}}\neq0$ 
(resp. $\wtil{\bV}\lambda^{\what{C}}=0$ and $\wtil{\bV}\wtil{\bV}\lambda^{\what{C}}\neq0$) 
at $p$ is equivalent to $\wtil{\bV}(\wtil{\bV}\hat{\rho})\neq0$ 
(resp. $\wtil{\bV}(\wtil{\bV}\hat{\rho})=0$ and 
$\wtil{\bV}\wtil{\bV}(\wtil{\bV}\hat{\rho})\neq0$) at $p$. 
By direct calculations, we see that $\wtil{\bV}(\wtil{\bV}\hat{\rho})(p)=e'(0)\hat{\kappa}(p)\lambda_v(p)=e'(0)\mu_c$ and 
\begin{equation*}
\wtil{\bV}\wtil{\bV}(\wtil{\bV}\hat{\rho})(p)=
-3e'(0)\hat{\kappa}(p)\wtil{V}_1(p)((\wtil{V}_1)_u(p)\hat{\rho}_v(p)+\wtil{V}_1(p)\hat{\rho}_{uv}(p))
-e''(0)\hat{\kappa}(p)\wtil{V}_1(p)^2\hat{\rho}_v(p)
\end{equation*}
hold. 
Thus we have the assertions by Fact \ref{fact:crit}, Lemma \ref{lem:swallow} and Proposition \ref{prop:frontness}.
\end{proof}

\begin{ex}\label{ex:sw-ce}
Let $f\colon\R^2\to\R^3$ be a map given by 
$$f(u,v)=\left(\dfrac{u^2}{2}-v,-\dfrac{u^3}{3}+u v,-\dfrac{u^4}{8}+\dfrac{u^2 v}{2}\right).$$
Then the set of singular points of $f$ is the $u$-axis, 
and $f$ has a swallowtail at the origin (see Figure \ref{fig:example}, left). 
Moreover, the null vector field $\eta$ of $f$ can be written as 
$\eta=\partial_u-u\partial_v$. 
The image of focal surface $\what{C}$ is shown in the center of Figure \ref{fig:example}. 
One can verify that $\what{C}$ has a cuspidal edge at the origin. 
%The null vector field $\eta$ for $f$ can be taken as 
%$\eta=\partial_u-u\partial_v$. 
\begin{figure}[htbp]
  \begin{center}
    \begin{tabular}{c}

      % 1
      \begin{minipage}{0.3\hsize}
        \begin{center}
          \includegraphics[clip, width=3cm]{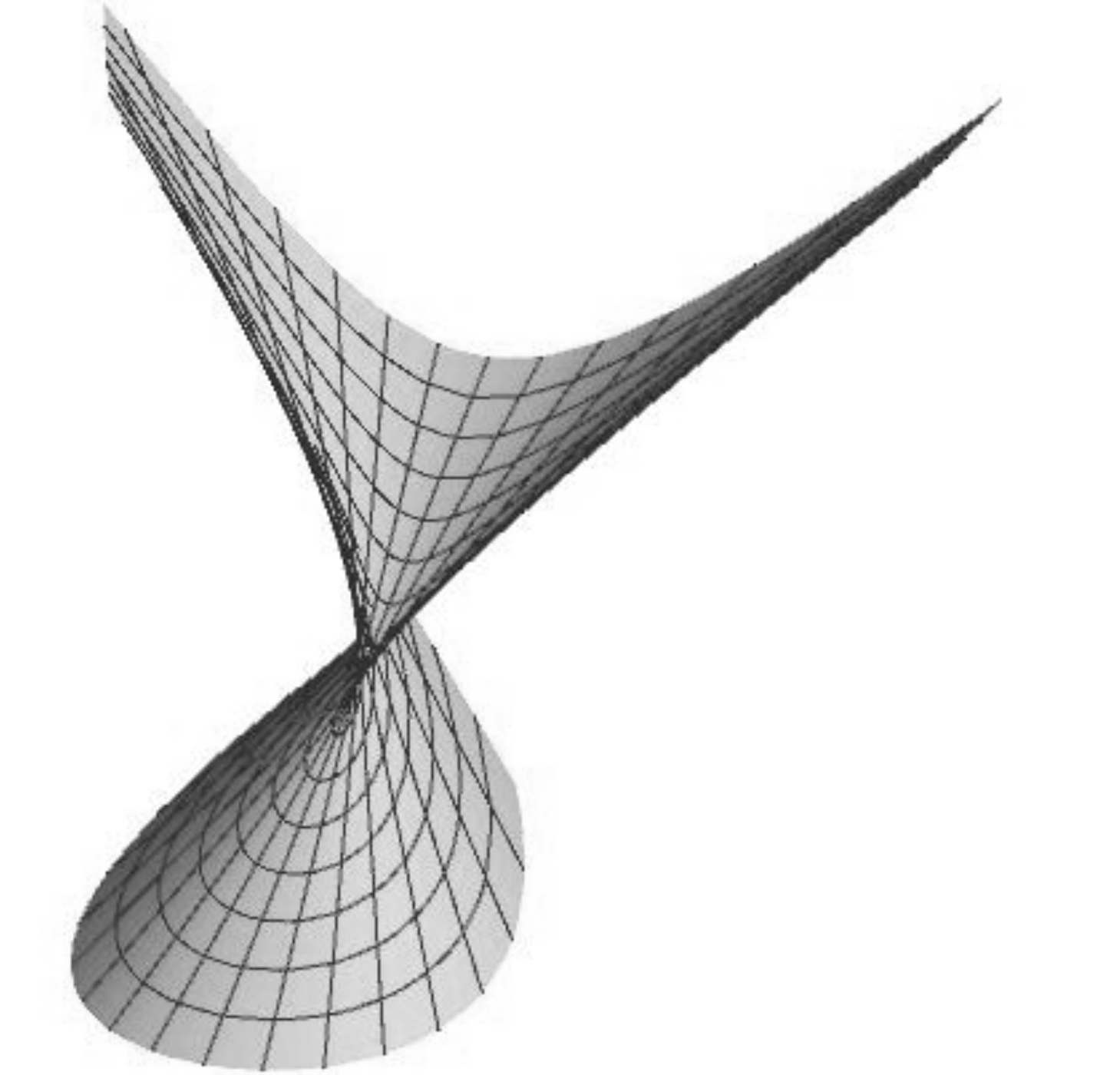}
        \end{center}
      \end{minipage}

\begin{minipage}{0.3\hsize}
        \begin{center}
          \includegraphics[clip, width=3cm]{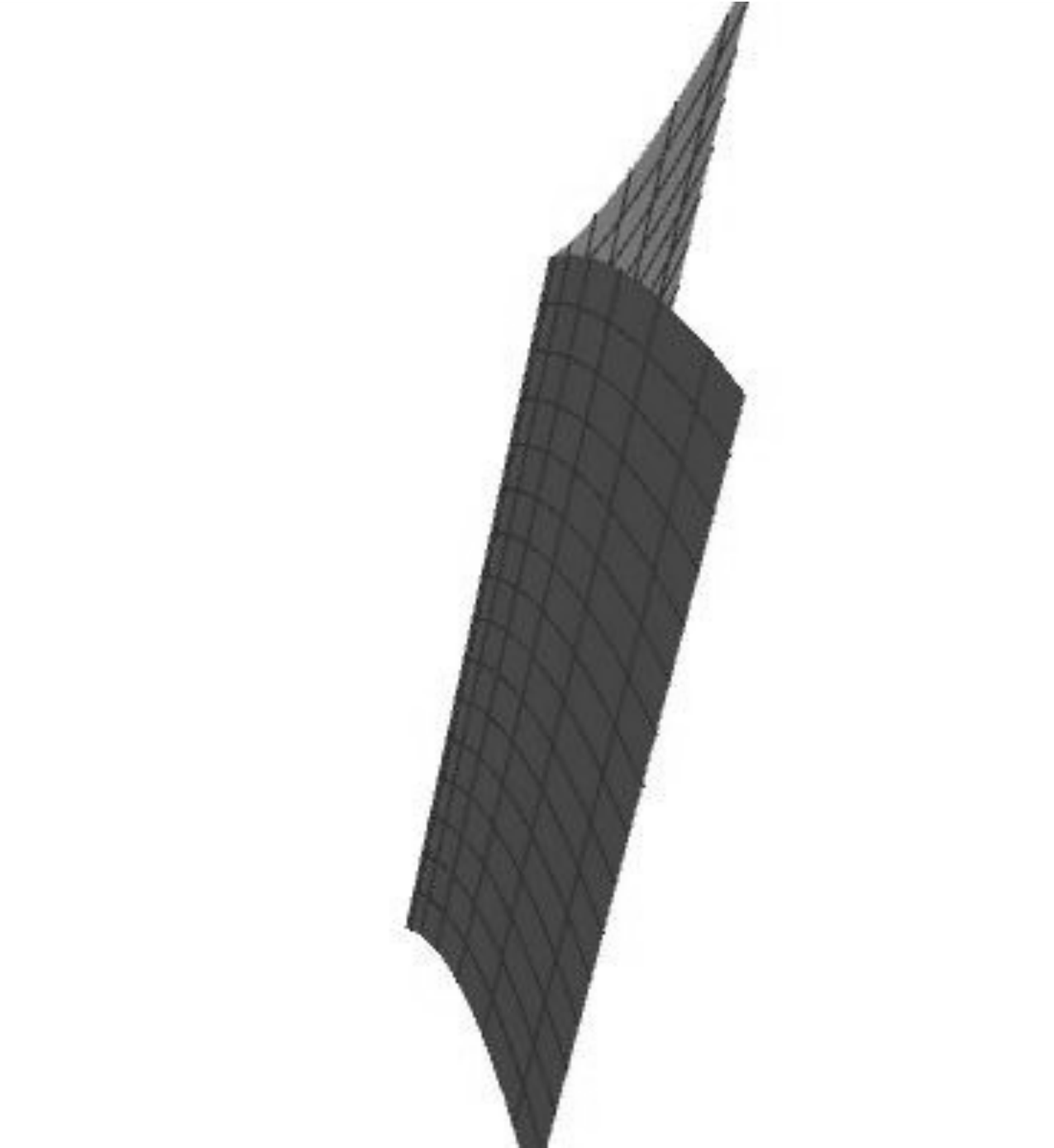}
        \end{center}
      \end{minipage}
      
      \begin{minipage}{0.3\hsize}
        \begin{center}
          \includegraphics[clip, width=3cm]{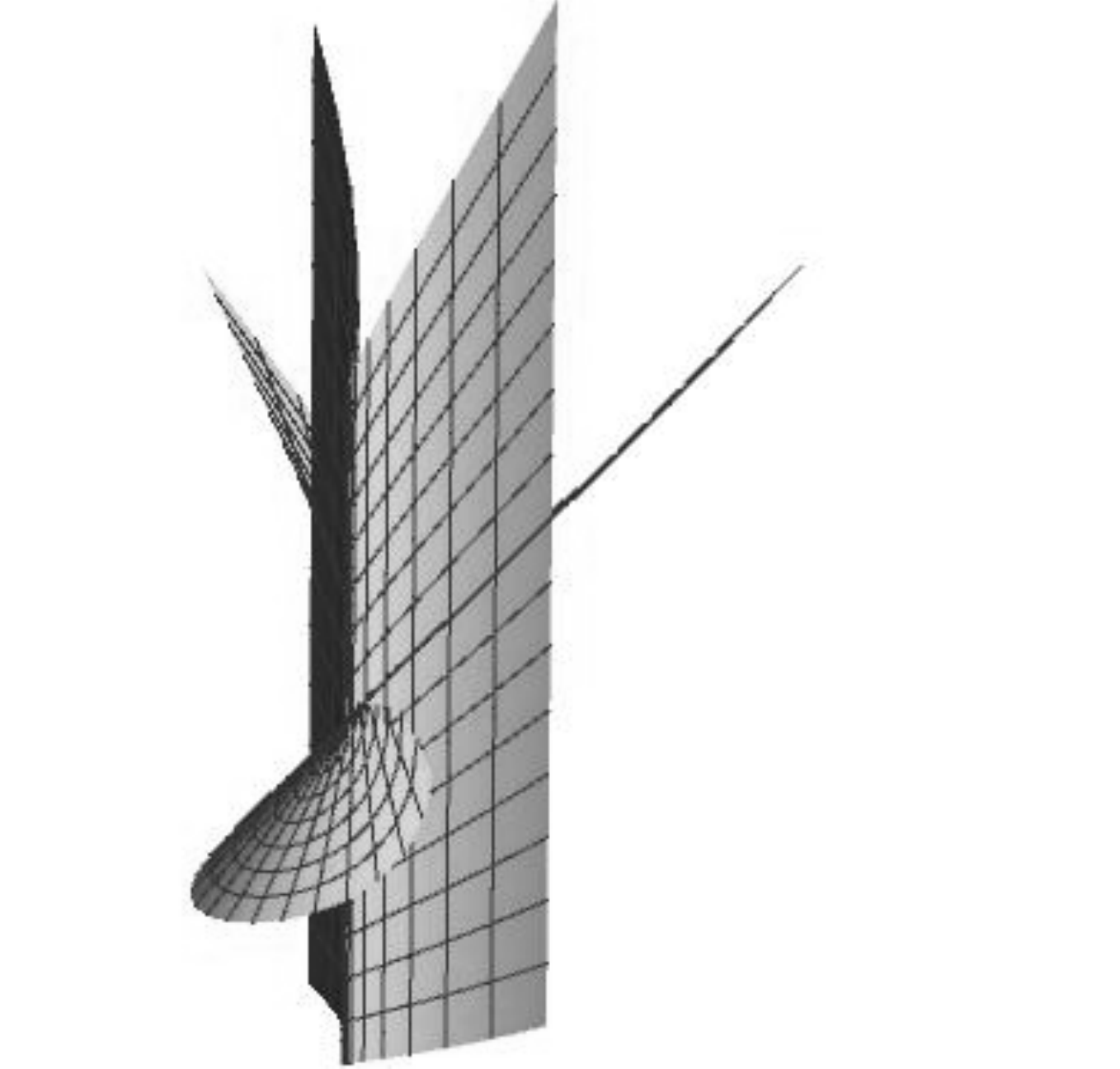}
        \end{center}
      \end{minipage}

    \end{tabular}
    \caption{The images of $f$ (left), $\what{C}$ (center) and both (right)
    of Example \ref{ex:sw-ce}.}
    \label{fig:example}
  \end{center}
\end{figure}
\end{ex}

\begin{ex}\label{ex:cbf-sw}
Let $f\colon\R^2\to\R^3$ be a $C^\infty$ map defined by 
\begin{align*}
f(u,v)=&\left(\frac{1}{6} \left(u^3-6 v\right),-\frac{u^4}{8}-\frac{u^3}{6}+u v+v,\right.\\
&\quad\quad\quad\left.\frac{1}{360} \left(-5 u^6-18 u^5+60 u^3 v+180 u^2 v-180 v^2\right)\right).
\end{align*}
The origin is a cuspidal butterfly of $f$ and $S(f)=\{v=0\}$. 
By a direct calculation, we have $(\wtil{\bV}\hat{\rho})_u=0$ and $(\wtil{\bV}\hat{\rho})_v=2(\neq0)$ at the origin. 
This implies that $\bV(\wtil{\bV}\hat{\rho})\neq0$ holds at the origin, 
and hence the origin is a non-degenerate singular point of the focal surface $\what{C}$ of $f$. 
Moreover, one can see that $\wtil{\bV}(\wtil{\bV}\hat{\rho})=0$ and $\wtil{\bV}\wtil{\bV}(\wtil{\bV}\hat{\rho})=-4\neq0$ hold at the origin. 
Thus $\what{C}$ has a swallowtail at the origin (see Figure \ref{fig:cbf_sw}). 
\begin{figure}[htbp]
  \begin{center}
    \begin{tabular}{c}

      % 1
      \begin{minipage}{0.3\hsize}
        \begin{center}
          \includegraphics[clip, width=3cm]{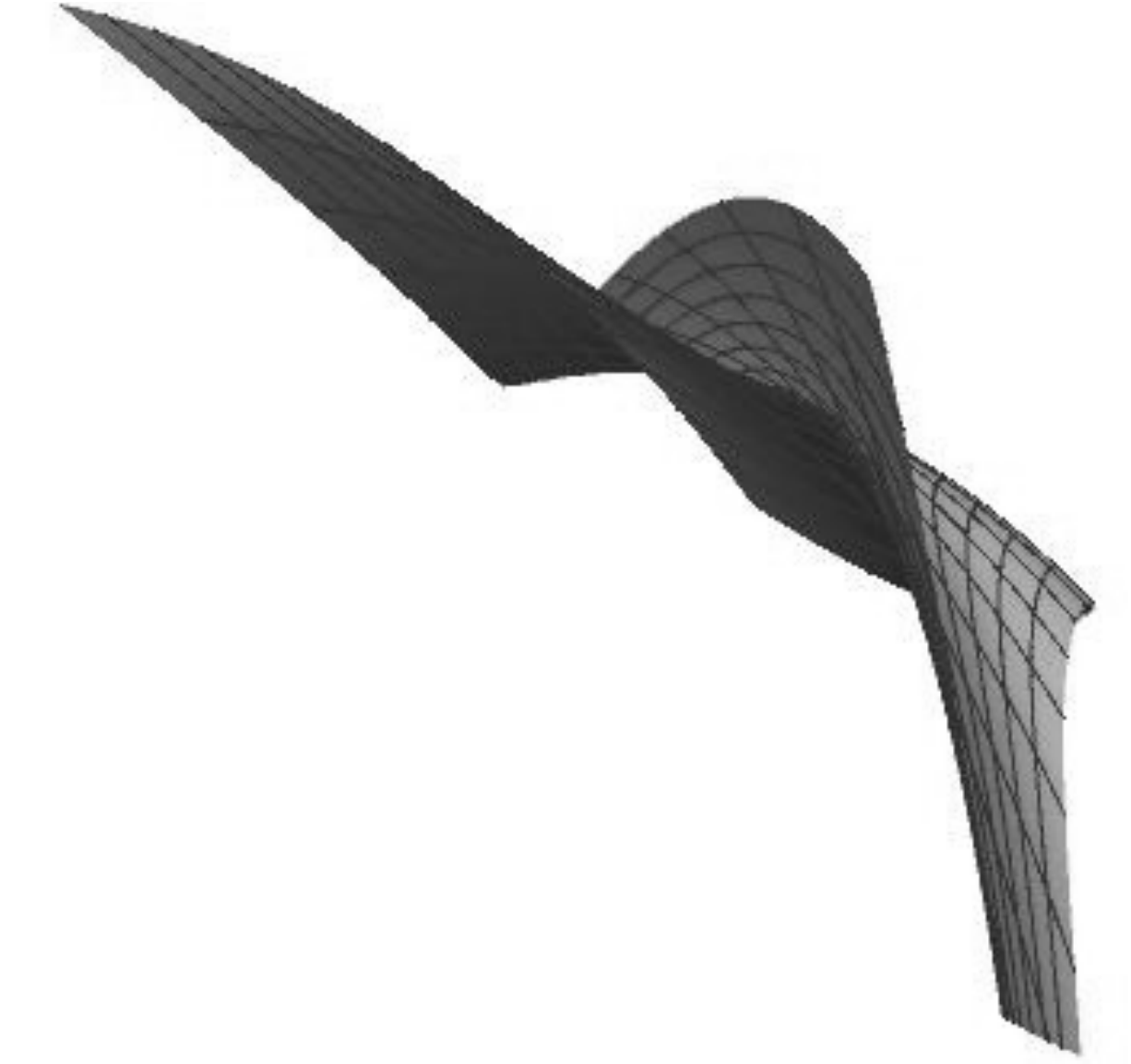}
        \end{center}
      \end{minipage}

      \begin{minipage}{0.3\hsize}
        \begin{center}
          \includegraphics[clip, width=3cm]{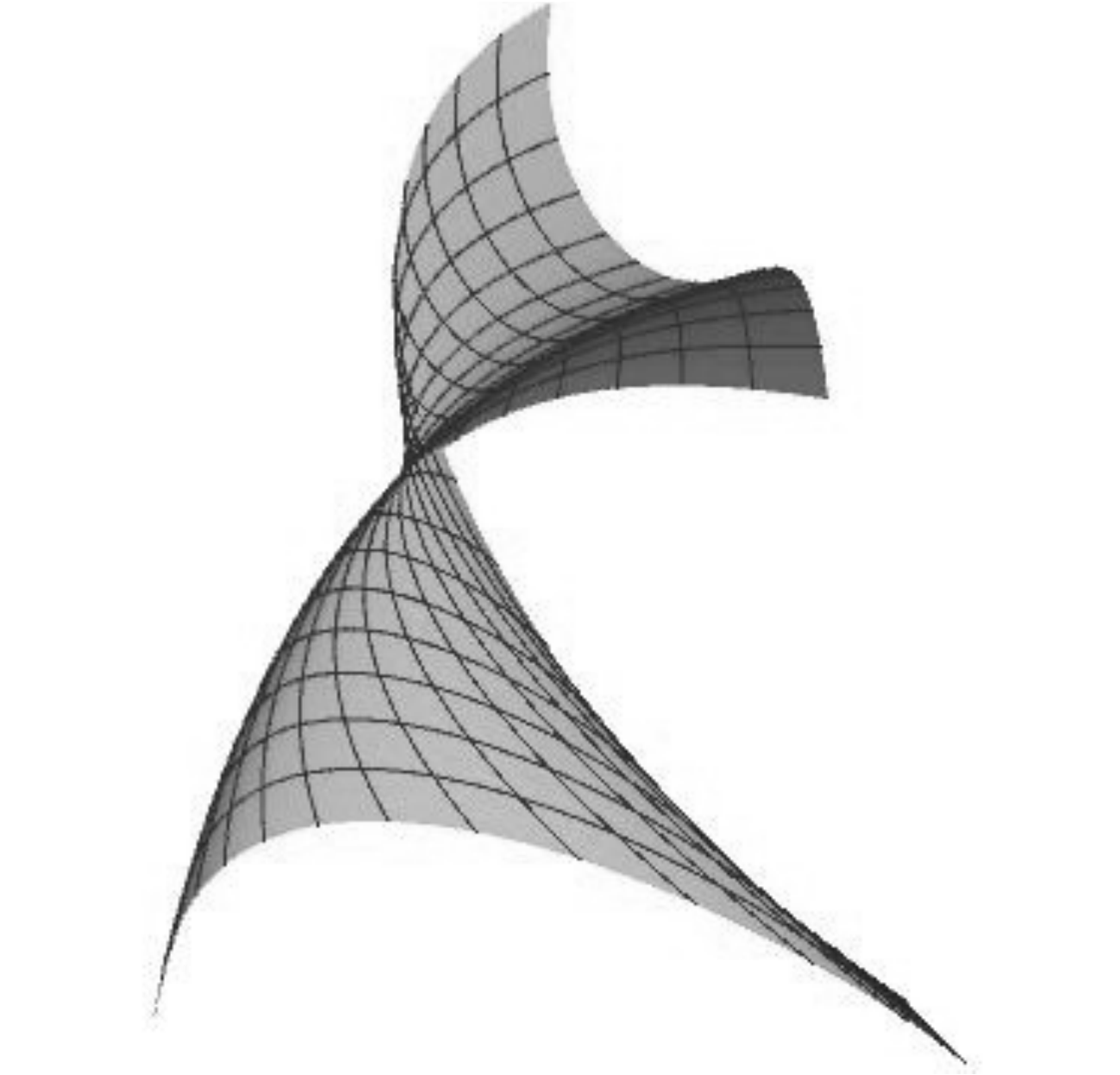}
        \end{center}
      \end{minipage}

\begin{minipage}{0.3\hsize}
        \begin{center}
          \includegraphics[clip, width=3cm]{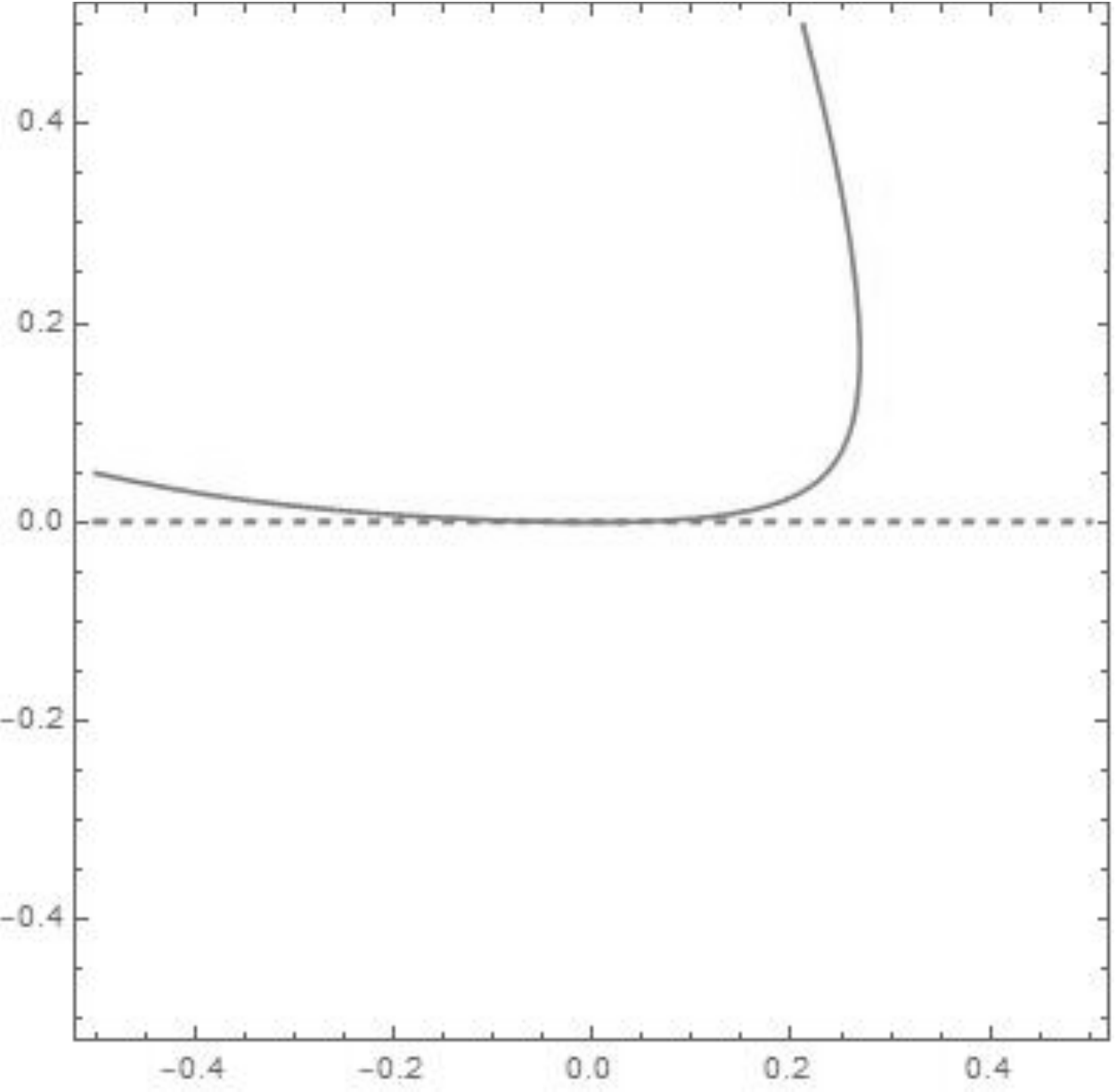}
        \end{center}
      \end{minipage}

    \end{tabular}
    \caption{Images of $f$ (left) and its focal surface $\what{C}$ (center) of Example \ref{ex:cbf-sw}.  
    The figure in the right-hand side shows the singular sets of both $f$ (dashed) and $\what{C}$ around the origin.}
    \label{fig:cbf_sw}
  \end{center}
\end{figure}

\end{ex}

As mentioned above, the focal surface $\what{C}$ given by \eqref{eq:focal} of a front $f$ 
might have a degenerate singularity at $p$ when $p$ is neither a cuspidal edge nor a swallowtail. 
By Theorem \ref{thm:singularity-C}, 
when $p$ is a cuspidal butterfly of $f$, 
then $\what{C}$ may have a cuspidal lips/beaks at $p$, 
which are degenerate singular points of fronts. 
%Here a {\it cuspidal lips} (resp. {\it cuspidal beaks}) is a germ 
%$\mathcal{A}$-equivalent to $(u,v)\mapsto(u,2v^3+u^2v,3v^4+u^2v^2)$ 
%(resp. $(u,v)\mapsto(u,-2v^3+u^2v,3v^4-u^2v^2)$) at the origin. 
%Criteria for these singular points are given in \cite[Theorem A.1]{ist-horoflat}.
Applying Fact \ref{fact:crit} for degenerate singularities, we have the following. 
\begin{prop}\label{prop:sing-C-deg}
Let $f\colon\Sig\to\R^3$ be a front and $p$ a cuspidal butterfly of $f$. 
Suppose that $\bV(\wtil{\bV}\hat{\rho})=0$ holds at $p$. 
Then the focal surface $\what{C}$ as in \eqref{eq:focal} 
is a cuspidal lips $($resp. cuspidal beaks$)$ at $p$ 
if and only if $\wtil{\bV}\hat{\rho}$ has a Morse type singularity of index zero or two 
$($resp. a Morse type singularity of index one$)$ at $p$. 
\end{prop}
\begin{proof}
By Lemma \ref{lem:swallow}, $p$ is a degenerate singular point of $\what{C}$. 
Moreover, by Theorem \ref{thm:singularity-C}, the condition $\wtil{\bV}\wtil{\bV}(\wtil{\bV}\hat{\rho})\neq0$ holds at $p$ 
when $p$ is a cuspidal butterfly of the initial front $f$. 
Thus we have the assertion by Fact \ref{fact:crit}. 
\end{proof}

\begin{rem}
Let $p$ be a singular point of the second kind of a front $f$ neither a swallowtail nor a cuspidal butterfly. 
Assume that the focal surface $\what{C}$ has a degenerate singularity at $p$. 
Then $\what{C}$ cannot have a cuspidal lips/beaks at $p$ 
because $\wtil{\bV}\wtil{\bV}(\wtil{\bV}\hat{\rho})=0$ holds at $p$ automatically. 
\end{rem}

\begin{ex}\label{ex:cbf-cbk}
Let $f\colon\R^2\to\R^3$ be a $C^\infty$ map given by 
\begin{equation*}
f(u,v)=\left(\frac{1}{6} \left(u^3-6 v\right),
\frac{1}{72} \left(u^6-9 u^4-12 u^3 v+72 u v+36 v^2\right),
-\frac{1}{20} u^2 \left(u^3-10 v\right)\right).
\end{equation*}
Then the set of singular points of $f$ is $S(f)=\{v=0\}$. 
Moreover, $f$ has a cuspidal butterfly at the origin. 
By direct calculations, we see that $(\wtil{\bV}\hat{\rho})_u=(\wtil{\bV}\hat{\rho})_v=0$ hold at the origin. 
Further, it follows that $(\wtil{\bV}\hat{\rho})_{uu}=-1,(\wtil{\bV}\hat{\rho})_{uv}=0$ and $(\wtil{\bV}\hat{\rho})_{vv}=6$ at the origin. 
This implies that $\wtil{\bV}\hat{\rho}$ has a Morse type singularity of index one at the origin. 
Thus the focal surface $\what{C}$ of $f$ has a cuspidal beaks at the origin (see Figure \ref{fig:cbf_beaks}). 
\begin{figure}[htbp]
  \begin{center}
    \begin{tabular}{c}

      % 1
      \begin{minipage}{0.3\hsize}
        \begin{center}
          \includegraphics[clip, width=3cm]{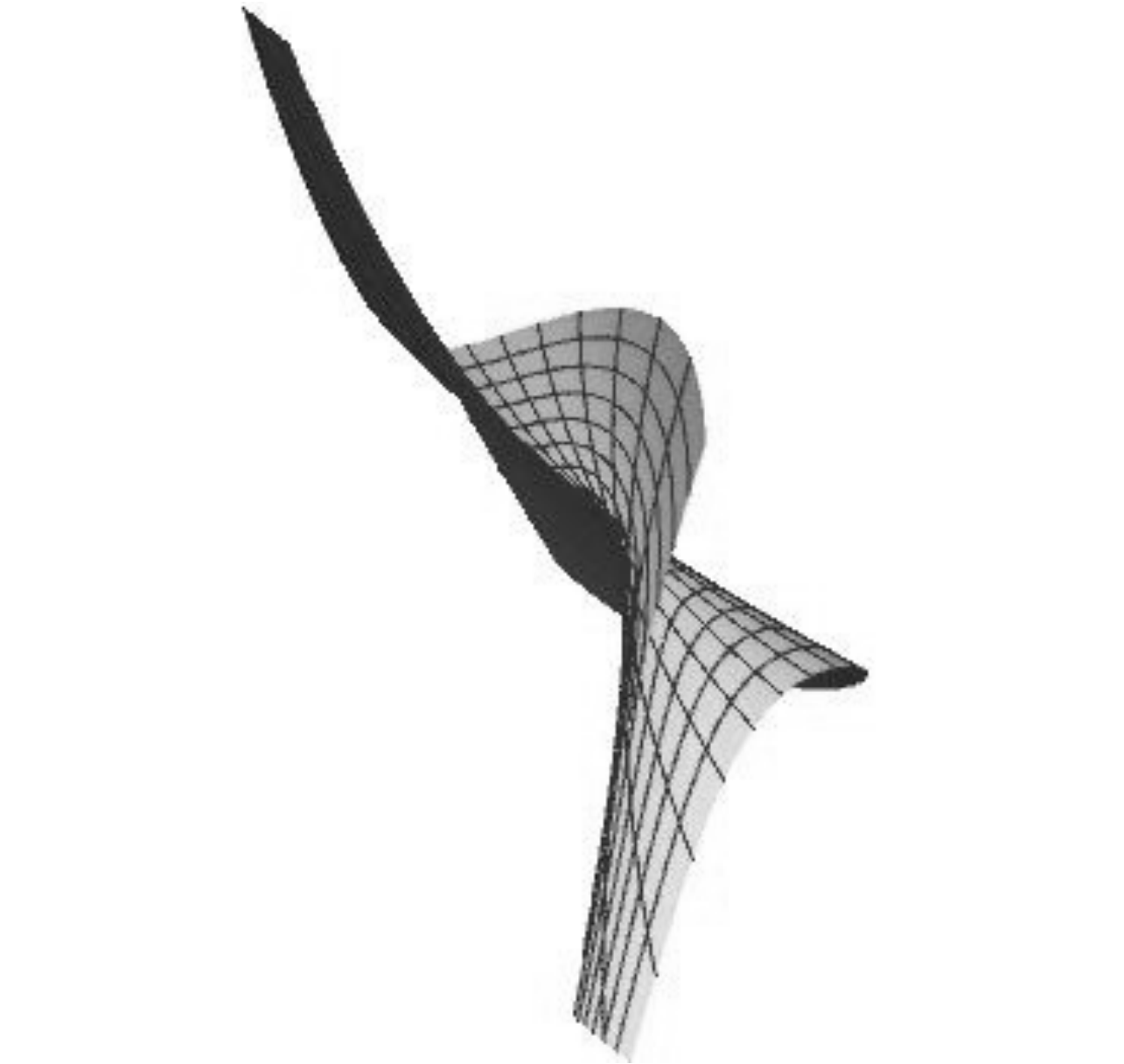}
        \end{center}
      \end{minipage}

      \begin{minipage}{0.3\hsize}
        \begin{center}
          \includegraphics[clip, width=3cm]{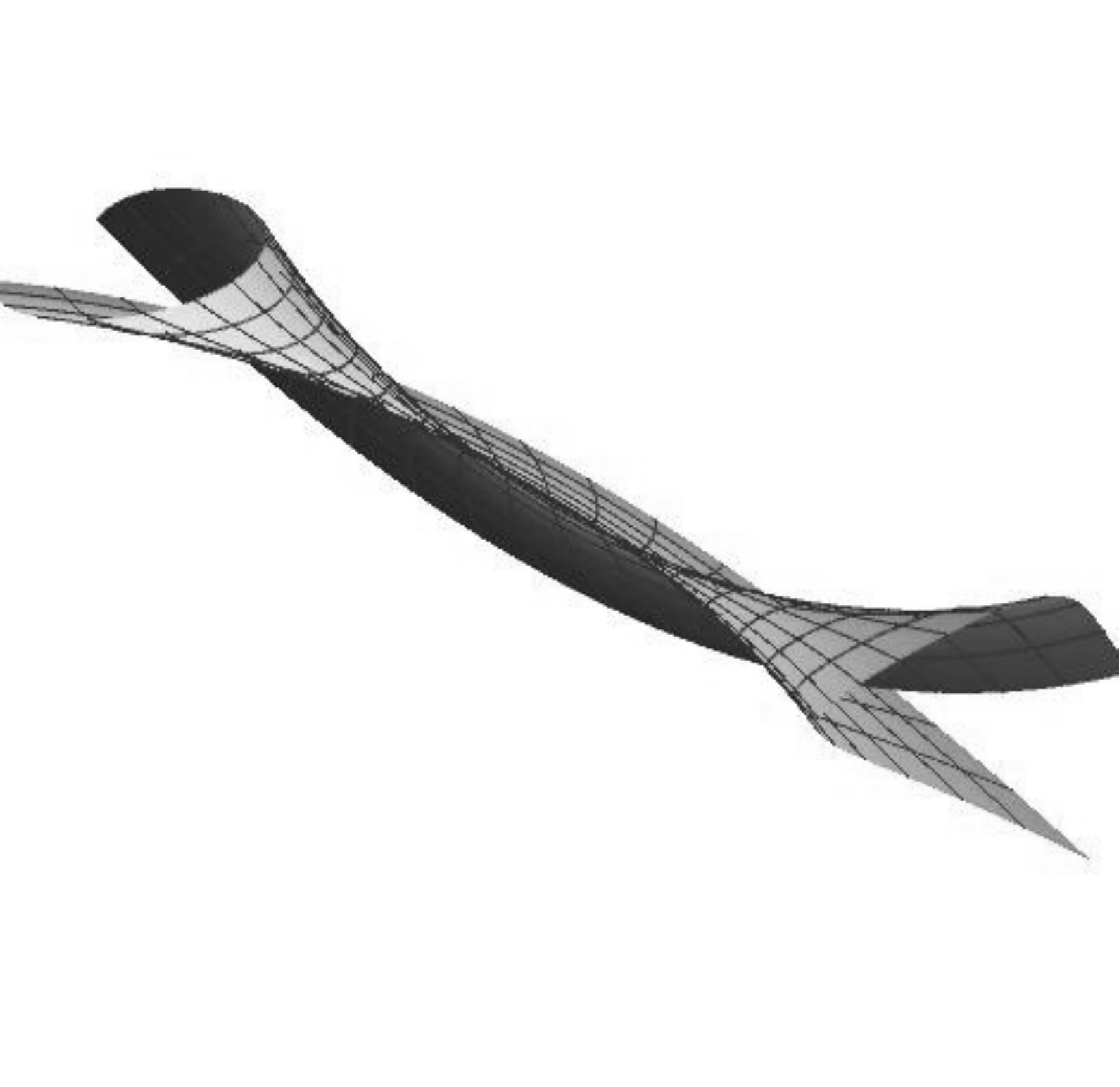}
        \end{center}
      \end{minipage}

\begin{minipage}{0.3\hsize}
        \begin{center}
          \includegraphics[clip, width=2.5cm]{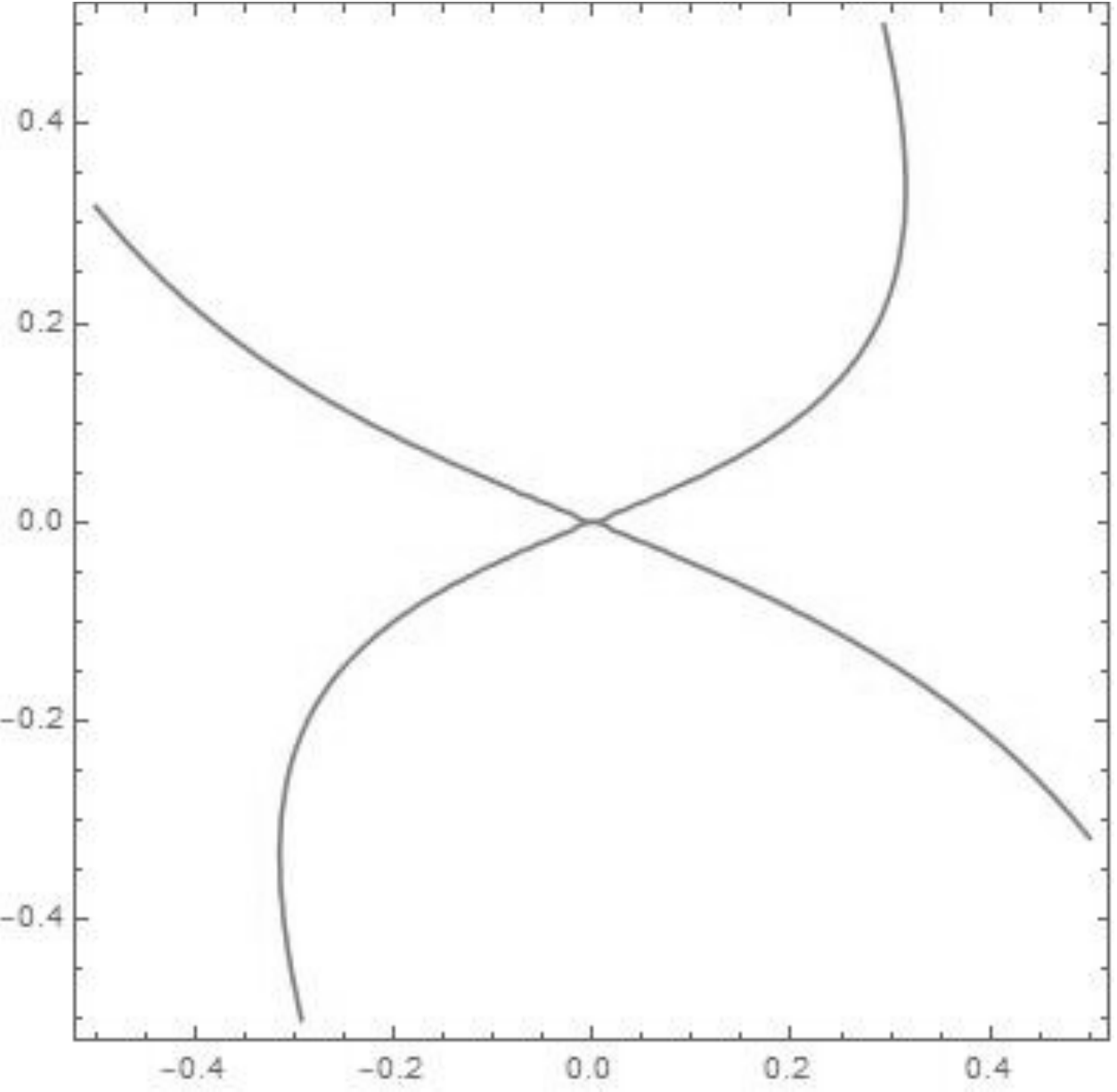}
        \end{center}
      \end{minipage}

    \end{tabular}
    \caption{Images of $f$ (left) and its focal surface $\what{C}$ (center) of Example \ref{ex:cbf-cbk}.  
    The figure in the right-hand side shows the singular set of $\what{C}$ around the origin.}
    \label{fig:cbf_beaks}
  \end{center}
\end{figure}
\end{ex}

\begin{ex}\label{ex:cbf-clp}
Let $f\colon\R^2\to\R^3$ be a $C^\infty$ map given by 
\begin{equation*}
f(u,v)=\left(\frac{1}{6} \left(u^3-6 v\right),
\frac{1}{72} \left(-u^6-9 u^4+12 u^3 v+72 u v-36 v^2\right),
-\frac{1}{20} u^2 \left(u^3-10 v\right)\right).
\end{equation*}
This map $f$ is a front and the origin is a cuspidal butterfly. 
By direct computations, it holds that $(\wtil{\bV}\hat{\rho})_u=(\wtil{\bV}\hat{\rho})_v=0$, 
$(\wtil{\bV}\hat{\rho})_{uu}=-1,(\wtil{\bV}\hat{\rho})_{uv}=0$ and $(\wtil{\bV}\hat{\rho})_{vv}=-6$ at the origin. 
This means that $\wtil{\bV}\hat{\rho}$ has a Morse type singularity of index two at the origin. 
Thus the focal surface $\what{C}$ has a cuspidal lips at the origin (see Figure \ref{fig:cbf_lips}). 
\begin{figure}[htbp]
  \begin{center}
    \begin{tabular}{c}

      % 1
      \begin{minipage}{0.33\hsize}
        \begin{center}
          \includegraphics[clip, width=3cm]{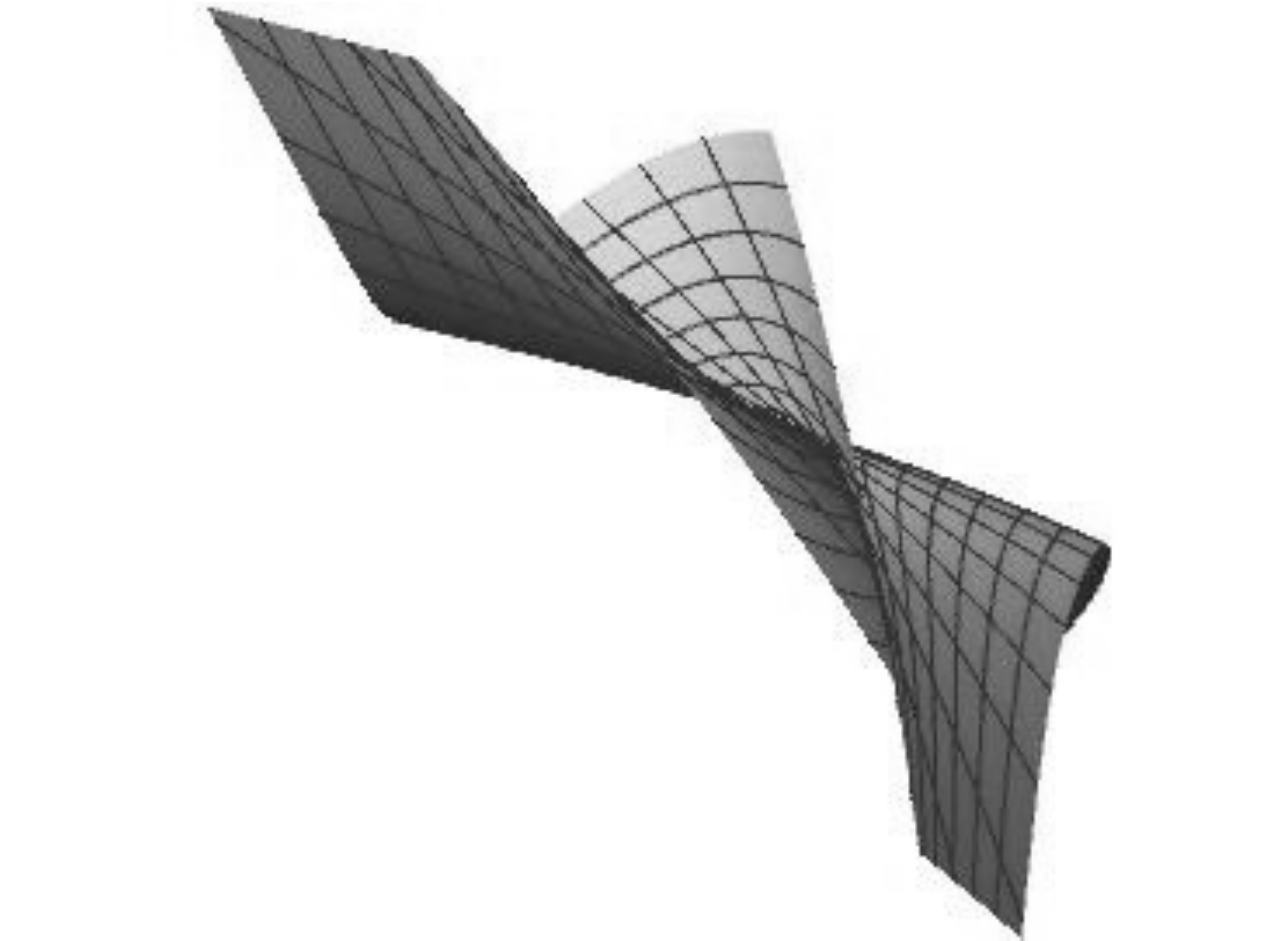}
        \end{center}
      \end{minipage}

      \begin{minipage}{0.33\hsize}
        \begin{center}
          \includegraphics[clip, width=3.5cm]{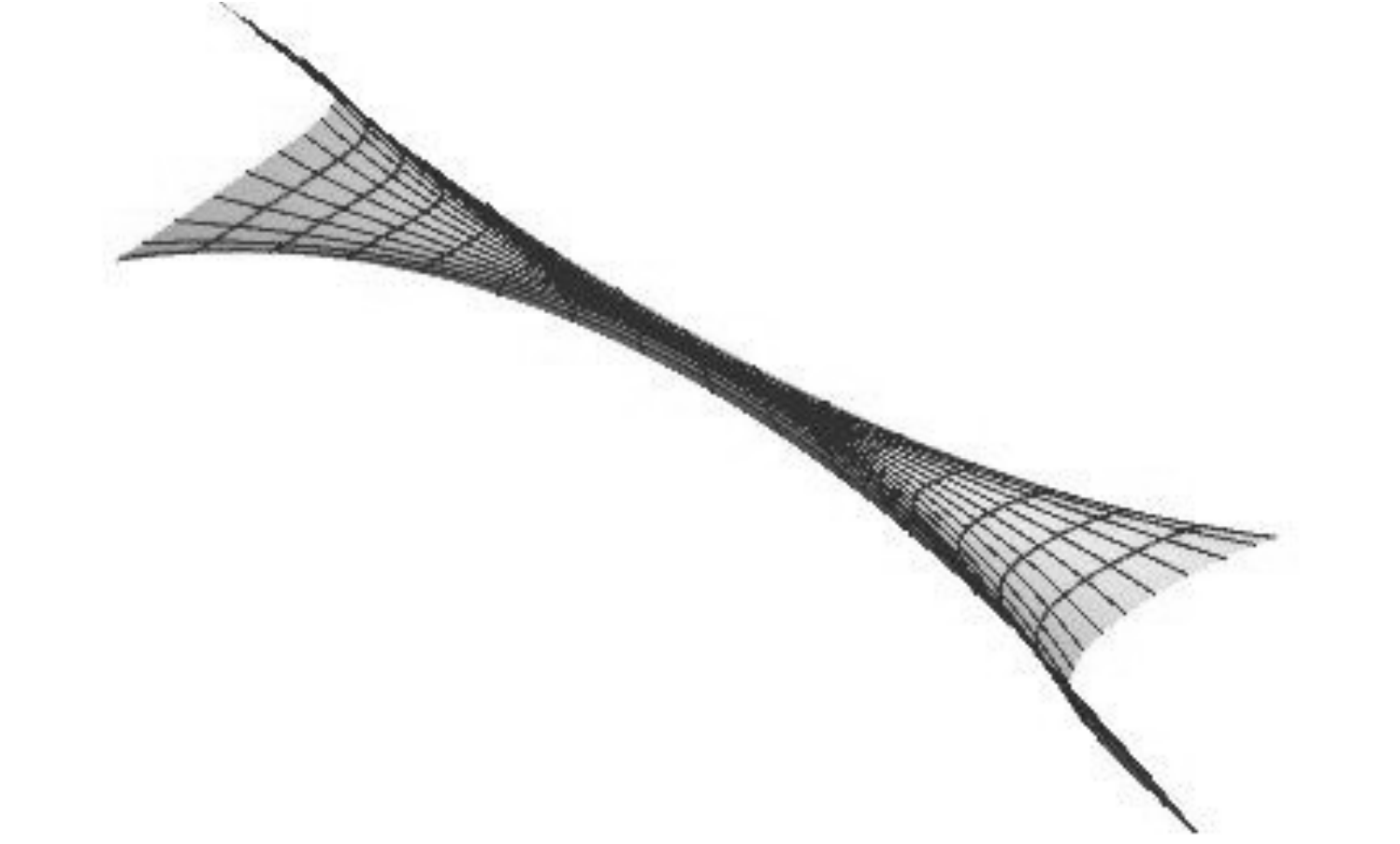}
        \end{center}
      \end{minipage}

    \end{tabular}
    \caption{The initial front $f$ (left) and its focal surface $\what{C}$ (right) of Example \ref{ex:cbf-clp}.}
    \label{fig:cbf_lips}
  \end{center}
\end{figure}

\end{ex}

\subsection{Contact between singular curves}
We now consider contact properties of singular curves of the initial front and its focal surface. 
First we start with the definition of contact between regular curves on the plane (cf. \cite[Page 74]{ifrt}). 
\begin{dfn}
Let $\alpha:I\ni t\mapsto\alpha(t)\in\R^2$ be a regular plane curve. 
Let $\beta$ be an another plane curve defined by the zero set of a smooth function $F:\R^2\to\R$. 
We say that {\it $\alpha$ has $(k+1)$-point contact at $t_0\in I$ with $\beta$} 
if the composite function $c(t)=F(\alpha(t))$ satisfies 
$$c(t_0)=c'(t_0)=\cdots=c^{(k)}(t_0)=0,\quad c^{(k+1)}(t_0)\neq0,$$
where $c^{(i)}=d^ic/dt^i$ $(1\leq i\leq k+1)$. 
Moreover, $\alpha$ has {\it at least $(k+1)$-point contact at $t_0$ with $\beta$} if 
the function $c(t)=F(\alpha(t))$ satisfies 
$$c(t_0)=c'(t_0)=\cdots=c^{(k)}(t_0)=0.$$ 
In this case, we call the integer $k$ the {\it order} of contact.
\end{dfn}

\begin{prop}\label{prop:contact}
Let $f\colon\Sig\to\R^3$ be a front and $p$ a singular point of the second kind. 
Assume that the curve given by $(\wtil{\bV}\hat{\rho})^{-1}(0)$ is regular at $p$. 
Then the singular curve $\gamma$ for $f$ through $\gamma(0)=p$ has 
$1$-point contact $($resp. $2$-point contact$)$ at $p$ with $(\wtil{\bV}\hat{\rho})^{-1}(0)$ if and only if 
$p$ is a swallowtail $($resp. cuspidal butterfly$)$ of $f$. 
Moreover, $\what{C}$ has a cuspidal edge $($resp. swallowtail$)$ at $p$. 
\end{prop}
\begin{proof}
Let us take a strongly adapted coordinate system $(U;u,v)$ centered at $p$. 
Define the function $c= \wtil{\bV}\hat{\rho}\circ\gamma\colon (-\eps,\eps)\to\R$. 
Since $\wtil{\bV}=\wtil{V}_1(\partial_u+e(u)\partial_v)$ and $\hat{\rho}=0$ on the $u$-axis, 
the function $c$ can be written as 
$$c(u)=e(u)\wtil{V}_1(u,0)\hat{\rho}_v(u,0)=e(u)\psi(u)\quad (\psi(u)=\wtil{V}_1(u,0)\hat{\rho}_v(u,0)).$$
By the Leibniz rule, one can see that 
\begin{equation}\label{eq:diff-g}
c^{(n)}(u)=\sum_{k=0}^n\binom{n}{k}~ e^{(k)}(u)\psi^{(n-k)}(u),
\end{equation}
where $\binom{n}{k}$ are binomial coefficients. 
Thus we have the assertions from Theorem \ref{thm:singularity-C}.
\end{proof}

The right-hand side of Figure \ref{fig:cbf_sw} shows both the sets of singular points of 
$f$ and $\what{C}$ of Example \ref{ex:cbf-sw}. 
We notice that these curves have $2$-point contact at the origin. 
This point actually corresponds to the swallowtail of $\what{C}$.

\subsection{Behavior of the Gaussian curvature of $\what{C}$}
Let $p$ be a singular point of the second kind of a front $f$. 
Then the focal surface $\what{C}$ given by \eqref{eq:focal} has a front-type singularity at $p$ by Theorem \ref{thm:singularity-C}. 
Thus the Gaussian curvature $K^{\what{C}}$ of $\what{C}$ may diverge at $p$ (\cite{suy-front,msuy}). 
However, $K^{\what{C}}$ might be rationally bounded at $p$. 
Hence we consider the condition for $K^{\what{C}}$ to be rationally bounded.

\begin{thm}\label{thm:r-bounded-focal}
Let $f\colon\Sig\to\R^3$ be a front and $p$ a singular point of the second kind, 
where $\Sig$ is a domain in $\R^2$. 
Let $\what{C}$ be the focal surface of $f$ associated to the unbounded principal curvature 
given by \eqref{eq:focal}. 
Suppose that $\what{C}$ has a non-degenerate singularity at $p$. 
Then the Gaussian curvature $K^{\what{C}}$ of $\what{C}$ is rationally bounded at $p$ 
if and only if $p$ is a sub-parabolic point of $f$. 
Moreover, the limiting normal curvature of $\what{C}$ vanishes at $p$. 
\end{thm}

\begin{proof}
Let us take a strongly adapted coordinate system $(U;u,v)$ centered at $p$. 
Then we consider the condition that the Gauss map of $\what{C}$ has a singularity at $p$. 
The Gauss map of $\what{C}$ is given by $\e_2=\y/|\y|$, 
where $\y$ is defined by \eqref{eq:vec-xy}. 
Define a function $\Lambda^{\what{C}}\colon U\to\R$ by 
$$\Lambda^{\what{C}}(u,v)=\det((\e_2)_u,(\e_2)_v,\e_2)(u,v).$$
Then zeros of $\Lambda^{\what{C}}$ correspond to singular points of $\e_2$. 
By a direct calculation, $\Lambda^{\what{C}}(p)=0$ is equivalent to 
$\det(\y_u,\y_v,\y)(p)=0$. 
Thus we consider this condition. 
By a direct calculation, we have 
$$\y=-\hat{\kappa}h,\quad 
\y_u=\ast_1 h-\hat{\kappa}h_u+\hat{\kappa}\what{F}_u f_v,\quad 
\y_v=\ast_2 h-\hat{\kappa}h_v-(\lambda_v\what{M}-\hat{\kappa}\what{F}_v)f_v$$
at $p$, where $\ast_i$ $(i=1,2)$ are some constants. 
Therefore one can see that 
\begin{equation*}
\det(\y_u,\y_v,\y)=-\hat{\kappa}\left(\hat{\kappa}^2\det(h_u,h_v,h)
+\hat{\kappa}(\lambda_v\what{M}-\hat{\kappa}\what{F}_v)\det(h_u,f_v,h)
-\hat{\kappa}^2\what{F}_u\det(f_v,h_v,h)\right)
\end{equation*}
holds at $p$. 
By Lemmas \ref{lem:derivative} and \ref{lem:ab}, 
we see that 
$$h_u=\dfrac{\what{E}_u}{2\what{E}}h+(\what{F}_u-\what{E})f_v+\what{L}\nu,\quad 
h_v=\dfrac{\what{E}_v}{2\what{E}}h+Bf_v+\what{M}\nu
$$
holds at $p$, where $B=\inner{h_v}{f_v}$. 
Since $\what{F}=\inner{f_v}{h}$, we have 
$\what{F}_v=\inner{f_{vv}}{h}+B$. 
On the other hand, since $\what{N}=\inner{f_{vv}}{\nu}$, 
it follows that $\what{N}_u=\inner{f_{uvv}}{\nu}+\inner{f_{vv}}{\nu_u}$. 
By $f_u=vh-e(u)f_v$, one can see that $f_{uvv}=2h_v$ holds at $p$. 
Moreover, $\nu_u=-\mu_c h$ holds at $p$ by Lemmas \ref{lem:weingarten2} and \ref{lem:relation}. 
Thus we obtain $\what{N}_u=2\what{M}-\mu_c\inner{f_{vv}}{h}$ at $p$, 
where we used the relation $\inner{h_v}{\nu}=\what{M}$. 
Hence we get 
$$B=\inner{f_{v}}{h_v}=\what{F}_v-\dfrac{2\what{M}-\what{N}_u}{\mu_c}$$
at $p$. 
Therefore, by $\mu_c=(\what{L}/\what{E})(p)$, it follows that 
\begin{align*}
\det(h_u,h_v,h)=\{(\what{F}_u\what{M}-\what{F}_v\what{L})+\what{E}(\what{M}-\what{N})\}\lambda_v,\quad
\det(h_u,f_v,h)=-\what{L}\lambda_v,\quad \det(f_v,h_v,h)=\what{M}\lambda_v
\end{align*}
hold at $p$. 
By the above calculations and $\hat{\kappa}=\mu_c\lambda_v$ at $p$, we have 
$$\det(\y_u,\y_v,\y)=-\hat{\kappa}^3\what{E}\lambda_v\what{N}_u$$
at $p$.

On the other hand, the directional derivative $\wtil{\bV}\kappa$ 
of the bounded principal curvature $\kappa$ in the direction $\wtil{\bV}$ is calculated as 
$$\wtil{\bV}\kappa=-\hat{\kappa}\what{N}_u=-\lambda_v\mu_c\what{N}_u$$
at $p$. 
Thus $\lambda_v\what{N}_u=-\wtil{\bV}\kappa/\mu_c$ holds at $p$, and hence we have 
$$\det(\y_u,\y_v,\y)=\lambda_v^5\mu_c^2(\wtil{\bV}\kappa)$$
at $p$ by $\lambda_v(p)^2=\what{E}(p)$. 
Therefore the Gauss map $\e_2$ of $\what{C}$ has a singularity at $p$ 
if and only if $p$ is a sub-parabolic point of $f$. 
Thus by Fact \ref{fact:rational-bounded}, we have the conclusion.
\end{proof}

Theorem \ref{thm:r-bounded-focal} 
gives geometrical meanings of a sub-parabolic point of a front with a singular point of the second kind. 
As a corollary of this theorem, we have the following relation 
between behavior of the Gaussian curvature of the initial front and of the focal surface. 

\begin{cor}\label{cor:K-rational}
Let $f\colon\Sig\to\R^3$ be a front and $p$ a singular point of the second kind of $f$. 
Suppose that the Gaussian curvature $K$ is either bounded on a sufficiently small neighborhood of $p$ 
or rationally continuous at $p$. 
Then the Gaussian curvature $K^{\what{C}}$ of $\what{C}$ is rationally bounded at $p$. 
\end{cor}
\begin{proof}
First we assume that $K$ is bounded near $p$. 
By Proposition \ref{prop:subpara}, $p$ is a sub-parabolic point of $f$. 
Thus by Theorem \ref{thm:r-bounded-focal}, $K^{\what{C}}$ is rationally bounded at $p$.

Next we suppose that $K$ is rationally continuous at $p$. 
By Proposition \ref{prop:subp-rational}, $p$ is a sub-parabolic point of $f$. 
Thus it holds that the Gaussian curvature $K^{\what{C}}$ of $\what{C}$ 
is rationally bounded at $p$ by Theorem \ref{thm:r-bounded-focal} again. 
Therefore we have the assertion.
\end{proof}

%%%%%APPENDIX%%%%% If needed
%\appendix

%\bibliographystyle{amsplain}
%\bibliographystyle{amsalpha}
%\nocite{*}
%\bibliography{reference}

\begin{thebibliography}{10}
\bibitem{arnold-redbook}
V.~I. Arnol'd, \emph{Singularities of caustics and wave fronts}, Mathematics
  and its Applications (Soviet Series), vol.~62, Kluwer Academic Publishers
  Group, Dordrecht, 1990. \MR{1151185}

\bibitem{agv}
V.~I. Arnol'd, S.~M. Gusein-Zade, and A.~N. Varchenko, \emph{Singularities of
  differentiable maps. {V}olume 1. classification of critical points, caustics
  and wave fronts}, Modern Birkh\"{a}user Classics, Birkh\"{a}user/Springer,
  New York, 2012, Translated from the Russian by Ian Porteous based on a
  previous translation by Mark Reynolds, Reprint of the 1985 edition.
  \MR{2896292}

\bibitem{bgt-crest}
J.~W. Bruce, P.~J. Giblin, and F.~Tari, \emph{Ridges, crests and sub-parabolic
  lines of evolving surfaces}, Internat. J. Computer Vision \textbf{18} (1996),
  no.~3, 195--210.

\bibitem{bgt-focal}
\bysame, \emph{Families of surfaces: focal sets, ridges and umbilics}, Math.
  Proc. Cambridge Philos. Soc. \textbf{125} (1999), no.~2, 243--268.
  \MR{1643790}

\bibitem{bw-folding}
J.~W. Bruce and T.~C. Wilkinson, \emph{Folding maps and focal sets},
  Singularity theory and its applications, {P}art {I} ({C}oventry, 1988/1989),
  Lecture Notes in Math., vol. 1462, Springer, Berlin, 1991, pp.~63--72.
  \MR{1129024}

\bibitem{fsuy-maximal}
S.~Fujimori, K.~Saji, M.~Umehara, and K.~Yamada, \emph{Singularities of maximal
  surfaces}, Math. Z. \textbf{259} (2008), no.~4, 827--848. \MR{2403743}

\bibitem{fukui2020}
T.~Fukui, \emph{Local differential geometry of cuspidal edge and swallowtail},
  Osaka J. Math. \textbf{57} (2020), no.~4, 961--992. \MR{4160343}

\bibitem{ft-framed}
T.~Fukunaga and M.~Takahashi, \emph{Framed surfaces in the {E}uclidean space},
  Bull. Braz. Math. Soc. (N.S.) \textbf{50} (2019), no.~1, 37--65. \MR{3935057}

\bibitem{hhnsuy}
M.~Hasegawa, A.~Honda, K.~Naokawa, K.~Saji, M.~Umehara, and K.~Yamada,
  \emph{Intrinsic properties of surfaces with singularities}, Internat. J.
  Math. \textbf{26} (2015), no.~4, 1540008, 34 pp. \MR{3338072}

\bibitem{hnuy-isom}
A.~Honda, K.~Naokawa, M.~Umehara, and K.~Yamada, \emph{Isometric deformations
  of wave fronts at non-degenerate singular points}, Hiroshima Math. J.
  \textbf{50} (2020), no.~3, 269--312. \MR{4184262}

\bibitem{hs-rhamphoid}
A.~Honda and K.~Saji, \emph{Geometric invariants of $5/2$-cuspidal edges},
  Kodai Math. J. \textbf{42} (2019), no.~3, 496--525. \MR{4025756}

\bibitem{ifrt}
S.~Izumiya, M.~C. Romero~Fuster, M.~A.~S. Ruas, and F.~Tari, \emph{Differential
  geometry from a singularity theory viewpoint}, World Scientific Publishing
  Co. Pte. Ltd., Hackensack, NJ, 2016. \MR{3409029}

\bibitem{is-mandala}
S.~Izumiya and K.~Saji, \emph{The mandala of {L}egendrian dualities for
  pseudo-spheres in {L}orentz-{M}inkowski space and ``flat'' spacelike
  surfaces}, J. Singul. \textbf{2} (2010), 92--127. \MR{2763021}

\bibitem{ist-horoflat}
S.~Izumiya, K.~Saji, and M.~Takahashi, \emph{Horospherical flat surfaces in
  hyperbolic 3-space}, J. Math. Soc. Japan \textbf{62} (2010), no.~3, 789--849.
  \MR{2648063}

\bibitem{ist-line}
S.~Izumiya, K.~Saji, and N.~Takeuchi, \emph{Singularities of line congruences},
  Proc. Roy. Soc. Edinburgh Sect. A \textbf{133} (2003), no.~6, 1341--1359.
  \MR{2027650}

\bibitem{ist-flat-ce}
\bysame, \emph{Flat surfaces along cuspidal edges}, J. Singul. \textbf{16}
  (2017), 73--100. \MR{3655304}

\bibitem{ku-flattori}
Y.~Kitagawa and M.~Umehara, \emph{Extrinsic diameter of immersed flat tori in
  {$S^3$}}, Geom. Dedicata \textbf{155} (2011), 105--140. \MR{2863896}

\bibitem{krsuy-flat}
M.~Kokubu, W.~Rossman, K.~Saji, M.~Umehara, and K.~Yamada, \emph{Singularities
  of flat fronts in hyperbolic space}, Pacific J. Math. \textbf{221} (2005),
  no.~2, 303--351. \MR{2196639}

\bibitem{kruy-flatfocal}
M.~Kokubu, W.~Rossman, M.~Umehara, and K.~Yamada, \emph{Flat fronts in
  hyperbolic 3-space and their caustics}, J. Math. Soc. Japan \textbf{59}
  (2007), no.~1, 265--299. \MR{2302672}

\bibitem{ms-ce}
L.~F. Martins and K.~Saji, \emph{Geometric invariants of cuspidal edges},
  Canad. J. Math. \textbf{68} (2016), no.~2, 445--462. \MR{3484374}

\bibitem{msuy}
L.~F. Martins, K.~Saji, M.~Umehara, and K.~Yamada, \emph{Behavior of {G}aussian
  curvature and mean curvature near non-degenerate singular points on wave
  fronts}, Geometry and topology of manifolds, Springer Proc. Math. Stat., vol.
  154, Springer, [Tokyo], 2016, pp.~247--281. \MR{3555987}

\bibitem{tito-2020}
T.~A. Medina-Tejeda, \emph{Extendibility and boundedness of invariants on
  singularities of wavefronts}, arxiv:2011.09511, 2020.

\bibitem{morris}
R.~Morris, \emph{The sub-parabolic lines of a surface}, The mathematics of
  surfaces, {VI} ({U}xbridge, 1994), Inst. Math. Appl. Conf. Ser. New Ser.,
  vol.~58, Oxford Univ. Press, New York, 1996, pp.~79--102. \MR{1430581}

\bibitem{mura-ume}
S.~Murata and M.~Umehara, \emph{Flat surfaces with singularities in {E}uclidean
  3-space}, J. Differential Geom. \textbf{82} (2009), no.~2, 279--316.
  \MR{2520794}

\bibitem{nuy-cuspidal}
K.~Naokawa, M.~Umehara, and K.~Yamada, \emph{Isometric deformations of cuspidal
  edges}, Tohoku Math. J. (2) \textbf{68} (2016), no.~1, 73--90. \MR{3476137}

\bibitem{os-folded}
R.~Oset~Sinha and K.~Saji, \emph{On the geometry of folded cuspidal edges},
  Rev. Mat. Complutense \textbf{31} (2018), no.~3, 627--650. \MR{3847079}

\bibitem{ot-flatce}
R.~Oset~Sinha and F.~Tari, \emph{On the flat geometry of the cuspidal edge},
  Osaka J. Math. \textbf{55} (2018), no.~3, 393--421. \MR{3824838}

\bibitem{porteous-normal}
I.~R. Porteous, \emph{The normal singularities of a submanifold}, J.
  Differential Geom. \textbf{5} (1971), no.~3-4, 543--564. \MR{292092}

\bibitem{porteous-book}
\bysame, \emph{Geometric differentiation. for the intelligence of curves and
  surfaces}, second ed., Cambridge University Press, Cambridge, 2001.
  \MR{1871900}

\bibitem{Roitman-flat}
P.~Roitman, \emph{Flat surfaces in hyperbolic space as normal surfaces to a
  congruence of geodesics}, Tohoku Math. J. (2) \textbf{59} (2007), no.~1,
  21--37. \MR{2321990}

\bibitem{saji-swallow}
K.~Saji, \emph{Normal form of the swallowtail and its applications}, Internat
  J. Math. \textbf{29} (2018), no.~7, 1850046, 17 pp. \MR{3825009}

\bibitem{st-behavior}
K.~Saji and K.~Teramoto, \emph{Behavior of principal curvatures of frontals
  near non-front singular points and their applications}, J. Geom. \textbf{112}
  (2021), no.~3, Paper No. 39, 25. \MR{4328072}

\bibitem{suy-ak}
K.~Saji, M.~Umehara, and K.~Yamada, \emph{{$A_k$} singularities of wave
  fronts}, Math. Proc. Cambridge Philos. Soc. \textbf{146} (2009), no.~3,
  731--746. \MR{2496355}

\bibitem{suy-front}
\bysame, \emph{The geometry of fronts}, Ann. of Math. (2) \textbf{169} (2009),
  no.~2, 491--529. \MR{2480610}

\bibitem{tera2}
K.~Teramoto, \emph{Focal surfaces of wave fronts in the {E}uclidean 3-space},
  Glasg. Math. J. \textbf{61} (2019), no.~2, 425--440. \MR{3928646}

\bibitem{tera3}
\bysame, \emph{Principal curvatures and parallel surfaces of wave fronts}, Adv.
  Geom. \textbf{19} (2019), no.~4, 541--554. \MR{4015189}

\bibitem{zak}
V.~M. Zakalyukin, \emph{Reconstructions of fronts and caustics depending on a
  parameter and versality of mappings}, J. Soviet Math. \textbf{27} (1984),
  2713--2735. \MR{0735440}

\end{thebibliography}

\end{document}